\documentclass[notitlepage]{amsart}

\theoremstyle{plain}
\newtheorem{thm}{\protect\theoremname}[section]
\theoremstyle{definition}
\newtheorem{defn}[thm]{\protect\definitionname}
\theoremstyle{plain}
\newtheorem{lem}[thm]{\protect\lemmaname}
\ifx\proof\undefined
\newenvironment{proof}[1][\protect\proofname]{\par
\normalfont\topsep6\p@\@plus6\p@\relax
\trivlist
\itemindent\parindent
\item[\hskip\labelsep\scshape #1]\ignorespaces
}{%
\endtrivlist\@endpefalse
}
\providecommand{\proofname}{Proof}
\fi
\theoremstyle{remark}
\newtheorem{rem}[thm]{\protect\remarkname}
\theoremstyle{plain}
\newtheorem{prop}[thm]{\protect\propositionname}
\theoremstyle{plain}
\newtheorem{cor}[thm]{\protect\corollaryname}
\usepackage{latexsym,amssymb,amscd} 
\usepackage{amsmath}
\usepackage[all]{xypic}
\usepackage{amsthm}

\usepackage[Lenny]{fncychap}
\usepackage{amscd}

\usepackage{tikz}

\usepackage{pgfplots}

\usetikzlibrary{arrows}

\usepackage{amssymb,amsmath}

\usepackage{graphicx}
\usepackage{pdflscape}
\usepackage{rotating} % Provides {sideways}{sidewaysfigure}{sidewaystable} environments

\usepackage{tikz}

\usepackage{pgfplots}

\usepackage{xargs}

\usetikzlibrary{arrows}

\usetikzlibrary{positioning}
\usetikzlibrary{backgrounds}
\usepackage{amssymb,amsmath}

\usepackage{enumitem}
\setenumerate{ label= (\alph*)}
\setenumerate[2]{ label= (\alph{enumi}\arabic*)} 
\makeatother

\providecommand{\corollaryname}{Corollary}
\providecommand{\definitionname}{Definition}
\providecommand{\lemmaname}{Lemma}
\providecommand{\propositionname}{Proposition}
\providecommand{\remarkname}{Remark}
\providecommand{\theoremname}{Theorem}

\newcommandx\tr[7][usedefault, addprefix=\global, 1=\eta, 2=X, 3=Y, 4=Z, 5=u, 6=v, 7=w]{#1:\quad#2\overset{#5}{\rightarrow}#3\overset{#6}{\rightarrow}#4\overset{#7}{\rightarrow}T#2}
\newcommandx\trp[7][usedefault, addprefix=\global, 1=\eta', 2=X', 3=Y', 4=Z', 5=u', 6=v', 7=w']{#1:\quad#2\overset{#5}{\rightarrow}#3\overset{#6}{\rightarrow}#4\overset{#7}{\rightarrow}T#2}
\newcommandx\tro[7][usedefault, addprefix=\global, 1=\eta, 2=X, 3=Y, 4=Z, 5=u, 6=v, 7=w]{#1=(#2,#3,#4,#5,#6,#7)}

\newcommandx\trop[7][usedefault, addprefix=\global, 1=\eta', 2=X', 3=Y', 4=Z', 5=u', 6=v', 7=w']{#1=(#2,#3,#4,#5,#6,#7)}

\newcommandx\catri[3][usedefault, addprefix=\global, 1=\mathcal{C}, 2=T, 3=\triangle]{\left(#1,#2,#3\right)}

\newcommandx\Cogennr[3][usedefault, addprefix=\global, 1=M, 2=n, 3=\mathcal{X}]{\operatorname{Cogen}_{#2}^{#3}(#1)}

\newcommandx\Gennr[3][usedefault, addprefix=\global, 1=M, 2=n, 3=\mathcal{X}]{\operatorname{Gen}_{#2}^{#3}(#1)}

\newcommandx\Hom[2][usedefault, addprefix=\global, 1=M, 2=N]{\operatorname{Hom}(#1{},#2{})}

\global\long\def\im#1{\operatorname{Im}#1}

\newcommandx\Add[1][usedefault, addprefix=\global, 1=\mathcal{M}]{\operatorname{Add}#1}

\newcommandx\p[2][usedefault, addprefix=\global, 1=\mathcal{A}, 2=\mathcal{B}]{\left(#1,#2\right)}

\newcommandx\ts[2][usedefault, addprefix=\global, 1=\mathcal{V}, 2=\mathcal{V}]{\left(#1^{\leq0},#2^{\geq0}\right)}

\newcommandx\cts[2][usedefault, addprefix=\global, 1=\mathcal{U}, 2=\mathcal{U}]{\left(#1_{\geq0},#2_{\leq0}\right)}

\newcommandx\Press[2][usedefault, addprefix=\global, 1=R, 2=\kappa]{\operatorname{Pres}_{\leq#2}\left(#1\right)}

\newcommandx\Summ[1][usedefault, addprefix=\global, 1=\mathcal{M}]{\operatorname{Summ}#1}

\newcommandx\thick[1][usedefault, addprefix=\global, 1=\mathcal{X}]{\operatorname{thick}\left(#1\right)}

\newcommandx\Coprod[1][usedefault, addprefix=\global, 1=\mathcal{M}]{\operatorname{Coprod}#1}

\newcommandx\Prodd[1][usedefault, addprefix=\global, 1=\mathcal{M}]{\operatorname{Prodd}#1}

\newcommandx\pd[1][usedefault, addprefix=\global, 1=M]{\operatorname{pd}#1}

\newcommandx\Gpd[3][usedefault, addprefix=\global, 1=M, 2=\mathcal{X}, 3=\mathcal{Y}]{\operatorname{Gpd}_{\left(#2,#3\right)}#1}

\newcommandx\WGpd[3][usedefault, addprefix=\global, 1=M, 2=\mathcal{X}, 3=\mathcal{Y}]{\operatorname{WGpd}_{\left(#2,#3\right)}#1}

\newcommandx\Gid[3][usedefault, addprefix=\global, 1=M, 2=\mathcal{X}, 3=\mathcal{Y}]{\operatorname{Gid}_{\left(#2,#3\right)}#1}

\newcommandx\id[1][usedefault, addprefix=\global, 1=M]{\operatorname{id}#1}

\newcommandx\pdr[2][usedefault, addprefix=\global, 1=M, 2=\mathcal{A}]{\operatorname{pd}_{#2}#1}
\newcommandx\idr[2][usedefault, addprefix=\global, 1=M, 2=\mathcal{B}]{\operatorname{id}_{#2}#1}

\newcommandx\Ker[1][usedefault, addprefix=\global, 1=M]{\operatorname{Ker}#1}
\newcommandx\Ann[1][usedefault, addprefix=\global, 1=M]{\operatorname{Ann}#1}

\renewcommandx\im[1][usedefault, addprefix=\global, 1=M]{\operatorname{Im}#1}
\newcommandx\Cok[1][usedefault, addprefix=\global, 1=M]{\operatorname{Coker}#1}

\newcommandx\Mod[1][usedefault, addprefix=\global, 1=R]{\operatorname{Mod}#1}

\newcommandx\Mor[1][usedefault, addprefix=\global, 1=R]{\operatorname{Mor}#1}

\newcommandx\Proj[1][usedefault, addprefix=\global, 1=R]{\operatorname{Proj}#1}

\newcommandx\Gp[2][usedefault, addprefix=\global, 1=\mathcal{X}, 2=\mathcal{Y}]{\mathcal{GP}_{(#1,#2)}}
\newcommandx\WGp[2][usedefault, addprefix=\global, 1=\mathcal{X}, 2=\mathcal{Y}]{\mathcal{WGP}_{(#1,#2)}}

\newcommandx\Gi[2][usedefault, addprefix=\global, 1=\mathcal{X}, 2=\mathcal{Y}]{\mathcal{GI}_{(#1,#2)}}
\newcommandx\WGi[2][usedefault, addprefix=\global, 1=\mathcal{X}, 2=\mathcal{Y}]{\mathcal{WGI}_{(#1,#2)}}

\newcommandx\nGp[2][usedefault, addprefix=\global, 1=\mathcal{X}, 2=\mathcal{Y}]{\mathcal{GP}_{(#1,#2)}^{<0}}
\newcommandx\nWGp[2][usedefault, addprefix=\global, 1=\mathcal{X}, 2=\mathcal{Y}]{\mathcal{WGP}_{(#1,#2)}^{<0}}

\newcommandx\nGi[2][usedefault, addprefix=\global, 1=\mathcal{X}, 2=\mathcal{Y}]{\mathcal{GI}_{(#1,#2)}^{<0}}
\newcommandx\nWGi[2][usedefault, addprefix=\global, 1=\mathcal{X}, 2=\mathcal{Y}]{\mathcal{WGI}_{(#1,#2)}^{<0}}

\newcommandx\proj[1][usedefault, addprefix=\global, 1=R]{\operatorname{proj}#1}

\newcommandx\Inj[1][usedefault, addprefix=\global, 1=R]{\operatorname{Inj}#1}

\newcommandx\inj[1][usedefault, addprefix=\global, 1=R]{\operatorname{inj}#1}

\newcommandx\suc[5][usedefault, addprefix=\global, 1=N, 2=M, 3=K, 4=, 5=]{0\rightarrow#1\overset{#4}{\rightarrow}#2\overset{#5}{\rightarrow}#3\rightarrow0}

\newcommandx\Cogen[1][usedefault, addprefix=\global, 1=M]{\operatorname{Cogen}#1}

\newcommandx\Pres[3][usedefault, addprefix=\global, 1=T, 2=\mathcal{D}, 3=n]{\operatorname{Pres}_{#2}^{#3}#1}

\newcommandx\Gen[1][usedefault, addprefix=\global, 1=M]{\operatorname{Gen}#1}
\newcommandx\Genn[2][usedefault, addprefix=\global, 1=M, 2=n]{\operatorname{Gen}_{#2}(#1)}

\newcommandx\Ext[4][usedefault, addprefix=\global, 1=M, 2=i, 3=R, 4=X]{\operatorname{Ext}_{#3}^{#2}\left(#1,#4\right)}

\newcommandx\Tor[4][usedefault, addprefix=\global, 1=M, 2=i, 3=R, 4=X]{\operatorname{Tor}_{#2}^{#3}\left(#1,#4\right)}

\renewcommandx\Hom[3][usedefault, addprefix=\global, 1=R, 2=M, 3=N]{\operatorname{Hom}{}_{#1}(#2,#3)}
\newcommandx\End[2][usedefault, addprefix=\global, 1=R, 2=M]{\operatorname{End}{}_{#1}(#2)}

\newcommandx\Homk[4][usedefault, addprefix=\global, 1=\mathcal{K}(R), 2=\sigma, 3=\omega, 4=1]{\operatorname{Hom}{}_{#1}(#2,#3[#4])}

\newcommandx\add[1][usedefault, addprefix=\global, 1=\mathcal{M}]{\operatorname{add}#1}

\newcommandx\soc[1][usedefault, addprefix=\global, 1=M]{\operatorname{soc}#1}

\newcommandx\rad[1][usedefault, addprefix=\global, 1=M]{\operatorname{rad}#1}

\newcommandx\K[1][usedefault, addprefix=\global, 1=R]{\mathcal{K}(#1)}
\newcommandx\D[1][usedefault, addprefix=\global, 1=R]{\mathcal{D}(#1)}

\newcommandx\pro[1][usedefault, addprefix=\global, 1=\mathcal{M}]{\operatorname{Prod}#1}

\renewcommandx\pd[1][usedefault, addprefix=\global, 1=M]{\operatorname{pd}#1}

\newcommandx\h[1][usedefault, addprefix=\global, 1=x]{\operatorname{hgt}\left(#1\right)}

\newcommandx\Findim[1][usedefault, addprefix=\global, 1=R_{R}]{\operatorname{Findim}#1}
\newcommandx\findim[1][usedefault, addprefix=\global, 1=R_{R}]{\operatorname{findim}#1}

\newcommandx\ortoder[1][usedefault, addprefix=\global, 1=\mathcal{M}]{#1^{\bot}}
\newcommandx\ortoizq[1][usedefault, addprefix=\global, 1=\mathcal{M}]{^{\bot}#1}

\newcommandx\ortoderi[2][usedefault, addprefix=\global, 1=\mathcal{M}, 2=i]{#1{}^{\bot_{#2}}}
\newcommandx\ortoizqi[2][usedefault, addprefix=\global, 1=\mathcal{M}, 2=i]{^{\bot_{#2}}#1}

\newcommandx\cores[3][usedefault, addprefix=\global, 1=N, 2=P, 3=f]{0\rightarrow#1\overset{#3_{0}}{\rightarrow}#2_{0}\overset{#3_{1}}{\rightarrow}...\overset{#3_{n}}{\rightarrow}#2_{n}\rightarrow...}

\newcommandx\res[3][usedefault, addprefix=\global, 1=N, 2=E, 3=f]{...\rightarrow#2^{n}\overset{#3_{n}}{\rightarrow}...\overset{#3_{1}}{\rightarrow}#2^{0}\overset{#3_{0}}{\rightarrow}#1\rightarrow0}

\renewcommandx\Add[1][usedefault, addprefix=\global, 1=\mathcal{M}]{\operatorname{Add}\left(#1\right)}

\renewcommandx\p[2][usedefault, addprefix=\global, 1=\mathcal{A}, 2=\mathcal{B}]{\left(#1,#2\right)}

\renewcommandx\Press[2][usedefault, addprefix=\global, 1=\mathcal{C}, 2=\kappa]{\operatorname{Pres}_{\leq#2}\left(#1\right)}

\renewcommandx\Summ[1][usedefault, addprefix=\global, 1=\mathcal{M}]{\operatorname{Summ}#1}

\renewcommandx\Coprod[1][usedefault, addprefix=\global, 1=\mathcal{M}]{\operatorname{Coprod}#1}

\renewcommandx\Prodd[1][usedefault, addprefix=\global, 1=\mathcal{M}]{\operatorname{Prodd}#1}

\renewcommandx\pd[1][usedefault, addprefix=\global, 1=M]{\operatorname{pd}\left(#1\right)}
\renewcommandx\id[1][usedefault, addprefix=\global, 1=M]{\operatorname{id}\left(#1\right)}

\renewcommandx\pdr[2][usedefault, addprefix=\global, 1=M, 2=\mathcal{A}]{\operatorname{pd}_{#2}\left(#1\right)}
\renewcommandx\idr[2][usedefault, addprefix=\global, 1=M, 2=\mathcal{B}]{\operatorname{id}_{#2}\left(#1\right)}

\renewcommandx\Ker[1][usedefault, addprefix=\global, 1=M]{\operatorname{Ker}#1}

\newcommandx\Keri[1][usedefault, addprefix=\global, 1=M]{\operatorname{Ker}^{-1}\left(#1\right)}

\renewcommandx\Ann[1][usedefault, addprefix=\global, 1=M]{\operatorname{Ann}#1}

\renewcommandx\im[1][usedefault, addprefix=\global, 1=M]{\operatorname{Im}\left(#1\right)}

\renewcommandx\Cok[1][usedefault, addprefix=\global, 1=M]{\operatorname{Coker}\left(#1\right)}

\newcommandx\Coki[1][usedefault, addprefix=\global, 1=M]{\operatorname{Coker}^{-1}\left(#1\right)}

\renewcommandx\Mod[1][usedefault, addprefix=\global, 1=R]{\operatorname{Mod}\left(#1\right)}

\renewcommandx\Mor[1][usedefault, addprefix=\global, 1=\mathcal{C}]{\operatorname{Mor}\left(#1\right)}

\renewcommandx\Proj[1][usedefault, addprefix=\global, 1=\mathcal{C}]{\operatorname{Proj}\left(#1\right)}

\renewcommandx\proj[1][usedefault, addprefix=\global, 1=\mathcal{C}]{\operatorname{proj}\left(#1\right)}

\renewcommandx\Inj[1][usedefault, addprefix=\global, 1=\mathcal{C}]{\operatorname{Inj}\left(#1\right)}

\renewcommandx\inj[1][usedefault, addprefix=\global, 1=\mathcal{C}]{\operatorname{inj}\left(#1\right)}

\renewcommandx\suc[5][usedefault, addprefix=\global, 1=N, 2=M, 3=K, 4=, 5=]{0\rightarrow#1\overset{#4}{\rightarrow}#2\overset{#5}{\rightarrow}#3\rightarrow0}

\renewcommandx\Cogen[1][usedefault, addprefix=\global, 1=M]{\operatorname{Cogen}\left(#1\right)}
\newcommandx\Cogenn[2][usedefault, addprefix=\global, 1=M, 2=n]{\operatorname{Cogen}_{#2}(#1)}
\renewcommandx\Cogennr[3][usedefault, addprefix=\global, 1=M, 2=n, 3=\mathcal{X}]{\operatorname{Cogen}_{#2}^{#3}(#1)}

\renewcommandx\Pres[1][usedefault, addprefix=\global, 1=M]{\operatorname{Pres}\left(#1\right)}

\renewcommandx\Gen[1][usedefault, addprefix=\global, 1=M]{\operatorname{Gen}\left(#1\right)}
\renewcommandx\Genn[2][usedefault, addprefix=\global, 1=M, 2=n]{\operatorname{Gen}_{#2}(#1)}
\renewcommandx\Gennr[3][usedefault, addprefix=\global, 1=M, 2=n, 3=\mathcal{X}]{\operatorname{Gen}_{#2}^{#3}(#1)}

\renewcommandx\Ext[4][usedefault, addprefix=\global, 1=M, 2=i, 3=\mathcal{C}, 4=X]{\operatorname{Ext}_{#3}^{#2}\left(#1,#4\right)}

\renewcommandx\Tor[4][usedefault, addprefix=\global, 1=M, 2=i, 3=\mathcal{C}, 4=X]{\operatorname{Tor}_{#2}^{#3}\left(#1,#4\right)}

\renewcommandx\Hom[3][usedefault, addprefix=\global, 1=\mathcal{C}, 2=M, 3=N]{\operatorname{Hom}{}_{#1}(#2,#3)}
\renewcommandx\End[2][usedefault, addprefix=\global, 1=\mathcal{C}, 2=M]{\operatorname{End}{}_{#1}(#2)}

\renewcommandx\Homk[4][usedefault, addprefix=\global, 1=\mathcal{K}(\mathcal{C}), 2=\sigma, 3=\omega, 4=1]{\operatorname{Hom}{}_{#1}(#2,#3[#4])}

\renewcommandx\add[1][usedefault, addprefix=\global, 1=\mathcal{M}]{\operatorname{add}#1}

\renewcommandx\soc[1][usedefault, addprefix=\global, 1=M]{\operatorname{soc}#1}

\renewcommandx\rad[1][usedefault, addprefix=\global, 1=M]{\operatorname{rad}#1}

\renewcommandx\K[1][usedefault, addprefix=\global, 1=\mathcal{C}]{\mathcal{K}(#1)}
\renewcommandx\D[1][usedefault, addprefix=\global, 1=R]{\mathcal{D}(#1)}

\renewcommandx\pro[1][usedefault, addprefix=\global, 1=\mathcal{M}]{\operatorname{Prod}#1}

\renewcommandx\pd[1][usedefault, addprefix=\global, 1=M]{\operatorname{pd}\left(#1\right)}

\renewcommandx\h[1][usedefault, addprefix=\global, 1=x]{\operatorname{hgt}\left(#1\right)}

\renewcommandx\Findim[1][usedefault, addprefix=\global, 1=\mathcal{C}]{\operatorname{Findim}\left(#1\right)}
\renewcommandx\findim[1][usedefault, addprefix=\global, 1=\mathcal{C}]{\operatorname{findim}\left(#1\right)}

\newcommandx\Findimr[2][usedefault, addprefix=\global, 1=\mathcal{C}, 2=\mathcal{X}]{\operatorname{Findim}_{#2}\left(#1\right)}
\newcommandx\findimr[2][usedefault, addprefix=\global, 1=\mathcal{C}, 2=\mathcal{X}]{\operatorname{findim}_{#2}\left(#1\right)}

\renewcommandx\cores[3][usedefault, addprefix=\global, 1=N, 2=P, 3=f]{0\rightarrow#1\overset{#3_{0}}{\rightarrow}#2_{0}\overset{#3_{1}}{\rightarrow}...\overset{#3_{n}}{\rightarrow}#2_{n}\rightarrow...}

\renewcommandx\res[3][usedefault, addprefix=\global, 1=N, 2=E, 3=f]{...\rightarrow#2^{n}\overset{#3_{n}}{\rightarrow}...\overset{#3_{1}}{\rightarrow}#2^{0}\overset{#3_{0}}{\rightarrow}#1\rightarrow0}

\renewcommandx\tr[7][usedefault, addprefix=\global, 1=\eta, 2=X, 3=Y, 4=Z, 5=u, 6=v, 7=w]{#1:\quad#2\overset{#5}{\rightarrow}#3\overset{#6}{\rightarrow}#4\overset{#7}{\rightarrow}T#2}
\renewcommandx\trp[7][usedefault, addprefix=\global, 1=\eta', 2=X', 3=Y', 4=Z', 5=u', 6=v', 7=w']{#1:\quad#2\overset{#5}{\rightarrow}#3\overset{#6}{\rightarrow}#4\overset{#7}{\rightarrow}T#2}
\renewcommandx\tro[7][usedefault, addprefix=\global, 1=\eta, 2=X, 3=Y, 4=Z, 5=u, 6=v, 7=w]{#1=(#2,#3,#4,#5,#6,#7)}

\renewcommandx\trop[7][usedefault, addprefix=\global, 1=\eta', 2=X', 3=Y', 4=Z', 5=u', 6=v', 7=w']{#1=(#2,#3,#4,#5,#6,#7)}

\renewcommandx\catri[3][usedefault, addprefix=\global, 1=\mathcal{C}, 2=T, 3=\triangle]{\left(#1,#2,#3\right)}

\renewcommandx\Cogennr[3][usedefault, addprefix=\global, 1=M, 2=n, 3=\mathcal{X}]{\operatorname{Cogen}_{#2}^{#3}(#1)}

\renewcommandx\Gennr[3][usedefault, addprefix=\global, 1=M, 2=n, 3=\mathcal{X}]{\operatorname{Gen}_{#2}^{#3}(#1)}

\renewcommandx\Hom[2][usedefault, addprefix=\global, 1=M, 2=N]{\operatorname{Hom}(#1{},#2{})}

\global\long\def\im#1{\operatorname{Im}#1}

\renewcommandx\Add[1][usedefault, addprefix=\global, 1=\mathcal{M}]{\operatorname{Add}#1}

\renewcommandx\p[2][usedefault, addprefix=\global, 1=\mathcal{A}, 2=\mathcal{B}]{\left(#1,#2\right)}

\renewcommandx\ts[2][usedefault, addprefix=\global, 1=\mathcal{V}, 2=\mathcal{V}]{\left(#1^{\leq0},#2^{\geq0}\right)}

\renewcommandx\cts[2][usedefault, addprefix=\global, 1=\mathcal{U}, 2=\mathcal{U}]{\left(#1_{\geq0},#2_{\leq0}\right)}

\renewcommandx\Press[2][usedefault, addprefix=\global, 1=R, 2=\kappa]{\operatorname{Pres}_{\leq#2}\left(#1\right)}

\renewcommandx\Summ[1][usedefault, addprefix=\global, 1=\mathcal{M}]{\operatorname{Summ}#1}

\renewcommandx\Coprod[1][usedefault, addprefix=\global, 1=\mathcal{M}]{\operatorname{Coprod}#1}

\renewcommandx\Prodd[1][usedefault, addprefix=\global, 1=\mathcal{M}]{\operatorname{Prodd}#1}

\renewcommandx\pd[1][usedefault, addprefix=\global, 1=M]{\operatorname{pd}#1}

\renewcommandx\Gpd[3][usedefault, addprefix=\global, 1=M, 2=\mathcal{X}, 3=\mathcal{Y}]{\operatorname{Gpd}_{\left(#2,#3\right)}#1}

\renewcommandx\WGpd[3][usedefault, addprefix=\global, 1=M, 2=\mathcal{X}, 3=\mathcal{Y}]{\operatorname{WGpd}_{\left(#2,#3\right)}#1}

\renewcommandx\Gid[3][usedefault, addprefix=\global, 1=M, 2=\mathcal{X}, 3=\mathcal{Y}]{\operatorname{Gid}_{\left(#2,#3\right)}#1}

\renewcommandx\id[1][usedefault, addprefix=\global, 1=M]{\operatorname{id}#1}

\renewcommandx\pdr[2][usedefault, addprefix=\global, 1=M, 2=\mathcal{A}]{\operatorname{pd}_{#2}\left(#1\right)}
\renewcommandx\idr[2][usedefault, addprefix=\global, 1=M, 2=\mathcal{B}]{\operatorname{id}_{#2}\left(#1\right)}

\renewcommandx\Ker[1][usedefault, addprefix=\global, 1=M]{\operatorname{Ker}#1}
\renewcommandx\Ann[1][usedefault, addprefix=\global, 1=M]{\operatorname{Ann}#1}

\renewcommandx\im[1][usedefault, addprefix=\global, 1=M]{\operatorname{Im}#1}
\renewcommandx\Cok[1][usedefault, addprefix=\global, 1=M]{\operatorname{Coker}#1}

\renewcommandx\Mod[1][usedefault, addprefix=\global, 1=R]{\operatorname{Mod}#1}

\renewcommandx\Mor[1][usedefault, addprefix=\global, 1=R]{\operatorname{Mor}#1}

\renewcommandx\Proj[1][usedefault, addprefix=\global, 1=R]{\operatorname{Proj}#1}

\renewcommandx\Gp[2][usedefault, addprefix=\global, 1=\mathcal{X}, 2=\mathcal{Y}]{\mathcal{GP}_{(#1,#2)}}

\newcommandx\Gpm[1][usedefault, addprefix=\global, 1=\mbox{Add}M]{\mathcal{GP}_{#1}}

\renewcommandx\WGp[2][usedefault, addprefix=\global, 1=\mathcal{X}, 2=\mathcal{Y}]{\mathcal{WGP}_{(#1,#2)}}

\newcommandx\WGpm[1][usedefault, addprefix=\global, 1=\mbox{Add}M]{\mathcal{WGP}_{#1}}

\renewcommandx\Gi[2][usedefault, addprefix=\global, 1=\mathcal{X}, 2=\mathcal{Y}]{\mathcal{GI}_{(#1,#2)}}

\newcommandx\Gim[1][usedefault, addprefix=\global, 1=\mbox{Add}M]{\mathcal{GI}_{#1}}

\renewcommandx\WGi[2][usedefault, addprefix=\global, 1=\mathcal{X}, 2=\mathcal{Y}]{\mathcal{WGI}_{(#1,#2)}}

\newcommandx\WGim[1][usedefault, addprefix=\global, 1=\mbox{Add}M]{\mathcal{WGI}_{#1}}

\renewcommandx\proj[1][usedefault, addprefix=\global, 1=R]{\operatorname{proj}#1}

\renewcommandx\Inj[1][usedefault, addprefix=\global, 1=R]{\operatorname{Inj}#1}

\renewcommandx\inj[1][usedefault, addprefix=\global, 1=R]{\operatorname{inj}#1}

\renewcommandx\suc[5][usedefault, addprefix=\global, 1=N, 2=M, 3=K, 4=, 5=]{0\rightarrow#1\overset{#4}{\rightarrow}#2\overset{#5}{\rightarrow}#3\rightarrow0}

\renewcommandx\Cogen[1][usedefault, addprefix=\global, 1=M]{\operatorname{Cogen}#1}

\renewcommandx\Pres[3][usedefault, addprefix=\global, 1=T, 2=\mathcal{D}, 3=n]{\operatorname{Pres}_{#2}^{#3}#1}

\renewcommandx\Gen[1][usedefault, addprefix=\global, 1=M]{\operatorname{Gen}#1}
\renewcommandx\Genn[2][usedefault, addprefix=\global, 1=M, 2=n]{\operatorname{Gen}_{#2}(#1)}

\renewcommandx\Ext[4][usedefault, addprefix=\global, 1=M, 2=i, 3=\mathcal{C}, 4=X]{\operatorname{Ext}_{#3}^{#2}\left(#1,#4\right)}

\renewcommandx\Tor[4][usedefault, addprefix=\global, 1=M, 2=i, 3=R, 4=X]{\operatorname{Tor}_{#2}^{#3}\left(#1,#4\right)}

\renewcommandx\Hom[3][usedefault, addprefix=\global, 1=\mathcal{C}, 2=M, 3=N]{\operatorname{Hom}{}_{#1}(#2,#3)}
\renewcommandx\End[2][usedefault, addprefix=\global, 1=R, 2=M]{\operatorname{End}{}_{#1}(#2)}

\renewcommandx\Homk[4][usedefault, addprefix=\global, 1=\mathcal{K}(R), 2=\sigma, 3=\omega, 4=1]{\operatorname{Hom}{}_{#1}(#2,#3[#4])}

\renewcommandx\add[1][usedefault, addprefix=\global, 1=\mathcal{M}]{\operatorname{add}#1}

\renewcommandx\soc[1][usedefault, addprefix=\global, 1=M]{\operatorname{soc}#1}

\renewcommandx\rad[1][usedefault, addprefix=\global, 1=M]{\operatorname{rad}#1}

\renewcommandx\K[1][usedefault, addprefix=\global, 1=R]{\mathcal{K}(#1)}
\renewcommandx\D[1][usedefault, addprefix=\global, 1=R]{\mathcal{D}(#1)}

\renewcommandx\pro[1][usedefault, addprefix=\global, 1=\mathcal{M}]{\operatorname{Prod}#1}

\renewcommandx\pd[1][usedefault, addprefix=\global, 1=M]{\operatorname{pd}\left(#1\right)}

\renewcommandx\h[1][usedefault, addprefix=\global, 1=x]{\operatorname{hgt}\left(#1\right)}

\renewcommandx\Findim[1][usedefault, addprefix=\global, 1=R_{R}]{\operatorname{Findim}#1}
\renewcommandx\findim[1][usedefault, addprefix=\global, 1=R_{R}]{\operatorname{findim}#1}

\renewcommandx\ortoder[1][usedefault, addprefix=\global, 1=\mathcal{M}]{#1^{\bot}}
\renewcommandx\ortoizq[1][usedefault, addprefix=\global, 1=\mathcal{M}]{^{\bot}#1}

\renewcommandx\ortoderi[2][usedefault, addprefix=\global, 1=\mathcal{M}, 2=i]{#1{}^{\bot_{#2}}}
\renewcommandx\ortoizqi[2][usedefault, addprefix=\global, 1=\mathcal{M}, 2=i]{^{\bot_{#2}}#1}

\renewcommandx\cores[3][usedefault, addprefix=\global, 1=N, 2=P, 3=f]{0\rightarrow#1\overset{#3_{0}}{\rightarrow}#2_{0}\overset{#3_{1}}{\rightarrow}...\overset{#3_{n}}{\rightarrow}#2_{n}\rightarrow...}

\renewcommandx\res[3][usedefault, addprefix=\global, 1=N, 2=E, 3=f]{...\rightarrow#2^{n}\overset{#3_{n}}{\rightarrow}...\overset{#3_{1}}{\rightarrow}#2^{0}\overset{#3_{0}}{\rightarrow}#1\rightarrow0}

\newcommandx\colim[1][usedefault, addprefix=\global, 1=F]{\operatorname{colim}#1}

\begin{document}

\title[Yoneda Ext]{The Yoneda Ext  and arbitrary coproducts in abelian categories}
%\thanks{}
\author[A. Argud{\'i}n Monroy]{Alejandro Argud{\'i}n Monroy}
%\date{}
\address{Instituto de Matem{\'a}ticas, Universidad Nacional Aut{\'o}noma de M{\'e}xico, Circuito Exterior, Ciudad Universitaria, C.P.04510 Mexico City, Mexico.}
\email{argudin@ciencias.unam.mx}

\thanks{The author thanks the Project PAPIIT-Universidad Nacional Aut\'onoma de M\'exico IN103317.}
\subjclass[2020]{Primary 18E99 , 18G15, 18A30}
\keywords{Extensions , Yoneda ext , coproducts , products , Ab4 categories}

\begin{abstract}
There are well known identities  involving the Ext bifunctor, coproducts,
and products in Ab4  abelian categories with enough projectives.
 Namely, for every such category $\mathcal{A}$, given
an object $X$ and a set of objects $\{ A_i \}_{i\in I}$, the following isomorphism can be built
$ \Ext[\bigoplus_{i\in I}A_{i}][n][\mathcal{A}]\cong\prod_{i\in I}\Ext[A_{i}][n][\mathcal{A}] $,
where $\operatorname{Ext}^n$  is the $n$-th derived functor of the Hom functor.
The goal of this paper is to show a similar isomorphism
for the $n$-th Yoneda Ext, which is a functor equivalent to $\operatorname{Ext}^n$ that can be defined in more general contexts.  The desired isomorphism is constructed
explicitly by using colimits,  in Ab4  abelian categories with not necessarily
enough projectives nor injectives, answering a question made by R. Colpi and K R. Fuller in \cite{colpi2007tilting}. Furthermore, the isomorphisms constructed are
used to characterize Ab4 categories. A dual result is also stated.
\end{abstract}
\maketitle

\section{Introduction }

The study of extensions is a theory that has developed from multiplicative
groups \cite{Schreirer1926,Holder1893}, with applications ranging
from representations of central simple algebras \cite{Brauer1928,Hasse1932}
to topology \cite{Eilenberg1942}.

In this article we will focus on extensions in an abelian category
$\mathcal{C}$. In this context, an extension of an object $A$ by
an object $C$ is a short exact sequence
\[
\suc[A][][C]
\]
up to equivalence, where two exact sequences are equivalent if there
is a morphism from one to another with identity morphisms at the ends.
This kind of approach was first made by R. Baer in 1934. On his work
\cite{Baer1934,Baer1935}, Baer defined an addition on the class $\Ext[C][1][\,][A]$
of extensions of an abelian group $A$ by an abelian  group $C$.
His construction can be easily extended to abelian categories, where
it is used to show that the class $\Ext[C][1][][A]$ has a natural
structure of abelian group. For this reason usually $\Ext[C][1][][A]$
is called the group of extensions of $A$ by $C$.

Later on, H. Cartan and S. Eilenberg \cite{Cartan-Eilemberg}, using
methods of homological algebra, showed that the first derived
functor of the $\Hom[\mathcal{C}][C][-]$ functor, or $\Hom[][-][A]$
functor, is isomorphic to $\Ext[C][1][][-]$, or respectively to $\Ext[-][1][\mathcal{C}][A]$.
This result marked the beginning of a series of research works looking
for ways of constructing the derived functors of the Hom functor without
using projective or injective objects, with the spirit that resolutions
should be only a calculation tool for derived functors. 

One of this attempts, registered in the work of D. Buchsbaum, B. Mitchell,
S. Schanuel, S. Mac Lane, M.C.R. Butler, and G. Horrocks \cite{maclanehomology,buchbaumext,Extrel,mitchell},
was based in the ideas of N. Yoneda \cite{Yoneda1954,Yoneda1960},
defining what is known today as the theory of $n$-extensions and
the functor called as the Yoneda Ext. An $n$-extension of an object
$A$ by an object $C$ is an exact sequence of length $n$
\[
\suc[A][M_{1}\rightarrow\cdots\rightarrow M_{n}][C]
\]
up to equivalence, where the equivalence of exact sequences of length
$n>1$ is defined in a similar way as was defined for length $1$.
In this theory, the Baer sum can be extended to $n$-extensions, proving
that the class $\Ext[C][n][\mathcal{C}][A]$ of $n$-extensions of
$A$ by $C$ is an abelian group.

Recently, the generalization of homological techniques such as Gorenstein
or tilting objects to abstract contexts \cite{tilcotiltcorresp,relgor,colpi2007tilting,vcoupek2017cotilting},
such as abelian categories that do not necessarily have projectives
or injectives, claim for the introduction of an Ext functor that can
be used without restraints. The only problem is that it is not clear if the rich
properties of the homological Ext are also valid for the
Yoneda Ext. The goal of this work is to make a next step by exploring
some properties that the Yoneda Ext shares with the homological Ext.

Namely, we will explore the following property that is well known
for module categories:
\begin{thm}
\cite[Proposition 7.21]{rotman} Let $R$ be a ring, $M\in\Mod$, and $\{N_{i}\}_{i\in I}$
be a set of $R$-modules. Then, there exist an isomorphism
\[
\Ext[\bigoplus_{i\in I}N_{i}][n][R][M]\cong\prod_{i\in I}\Ext[N_{i}][n][R][M]\mbox{.}
\]

\end{thm}
The proof of such theorem  can be extended to Ab4  abelian
categories with enough projectives. Our goal
will be to prove an analogue result for the Yoneda Ext without assuming
the existence of enough projectives.

Let us now describe the contents of this paper. Section 2 is devoted
to review the basic results of the theory of extensions by following
the steps of B. Mitchell in \cite{mitchell}. In section
3 we prove the desired theorem. More precisely, we show that in an
Ab4 abelian category we can build the desired bijections explicitly
by using colimits. Finally, in section 4 we use the bijections constructed in section 3
to characterize Ab4 categories.

\section{Extensions}

In this section we will remember the basic theory of extensions. As
was mentioned before, the theory of $n$-extensions was created by
Nobuo Yoneda in \cite{Yoneda1954}. In such paper he worked in a category
of modules and most of the results are related with the homological
tools built by projective and injective modules. Since our goal is
to work in an abelian category without depending on the existence
of projective or injective objects, we refer the reader to the work
of Barry Mitchell \cite{mitchell} for an approach in abelian categories
without further assumptions. Throughout this paper, $\mathcal{C}$
will denote an abelian category. 
\begin{defn}
\cite[Section 1]{mitchell} Let $C\in\mathcal{C}$, and $\alpha:A\rightarrow B$,
$\alpha':A'\rightarrow B'$ be morphisms in $\mathcal{C}$. We set the
following notation:
\begin{enumerate}
\item $\nabla_{C}:=\left(\begin{smallmatrix}1_{C} & 1_{C}\end{smallmatrix}\right):C\oplus C\rightarrow C\mbox{;}$
\item $\Delta_{C}:=\left(\begin{smallmatrix}1_{C}\\
1_{C}
\end{smallmatrix}\right):C\rightarrow C\oplus C\mbox{;}$
\item $\alpha\oplus\alpha':=\left(\begin{smallmatrix}\alpha & 0\\
0 & \alpha'
\end{smallmatrix}\right):A\oplus A'\rightarrow B\oplus B'\mbox{.}$
\end{enumerate}
\end{defn}

\subsection{1-Extensions}

Let us begin by recalling some basic facts and notation about 1-extensions.
\begin{defn}
\cite[Section 1]{mitchell} Let $\alpha:N\rightarrow N'$, $\beta:M\rightarrow M'$,
and $\gamma:K\rightarrow K'$ be morphisms in $\mathcal{C}$, and consider the following 
 short exact sequences in $\mathcal{C}$
 \[
 \eta:\;\suc[][][][f][g]\mbox{ and }\eta':\;\suc[N'][M'][K'][f'][g']\mbox{.}
\]
\begin{enumerate}
\item We say that $(\alpha,\beta,\gamma):\eta\rightarrow\eta'$  is a morphism
of short exact sequences if 
\[
\beta f=f'\alpha\;\mbox{and}\;\gamma g=g'\beta\mbox{.}
\]

\item We denote by $\eta\oplus\eta'$ to the short exact sequence 
\[
\suc[N\oplus N'][M\oplus M'][K\oplus K'][f\oplus f'][g\oplus g']\mbox{.}
\]

\end{enumerate}
\end{defn}

\begin{defn}
\cite[Section 1]{mitchell} For $N,K\in\mathcal{C}$, let $\mathcal{E}_{\mathcal{C}}(K,N)$
denote the class of short exact sequences of the form $\suc\mbox{.}$\end{defn}
\begin{rem}
Let $A,C\in\mathcal{C}$ and $\eta,\eta'\in\mathcal{E}_{\mathcal{C}}(C,A)$.
Consider the relation $\eta\equiv\eta'$ given by the existence of
a short exact sequence morphism $(1_{A},\beta,1_{C}):\eta\rightarrow\eta'\mbox{.}$
By the snake lemma, we know that $\beta$ is an isomorphism, and hence
$\equiv$ is an equivalence relation on $\mathcal{E}_{\mathcal{C}}(C,A)$.\end{rem}
\begin{defn}
\cite[Section 1]{mitchell} Consider $A,C\in\mathcal{C}$. 
\begin{enumerate}
\item Let $\Ext[C][1][][A]:=\mathcal{E}_{\mathcal{C}}(C,A)/\equiv\mbox{;}$
\item Each object of $\Ext[C][1][][A]$ is refered as an extension from
$A$ to $C$.
\item Every extension from $A$ to $C$ will be denoted with a capital letter
$E$, or by $\overline{\eta}$, in case $\eta$ is a representative
of the class $E$.
\item Given $\overline{\eta}\in\Ext[C][1][][A]$ and $\overline{\eta'}\in\Ext[C'][1][][A']$,
we will call extension morphism from $\overline{\eta}$ to $\overline{\eta'}$,
to every short exact sequence morphism $\eta\rightarrow\eta'$.
\item If $(\alpha,\beta,\gamma):E\rightarrow E'$ and $(\alpha',\beta',\gamma'):E'\rightarrow E''$
are extension morphisms, we define the composition morphism as 
\[
(\alpha',\beta',\gamma')(\alpha,\beta,\gamma):=(\alpha'\alpha,\beta'\beta,\gamma'\gamma)\mbox{.}
\]

\end{enumerate}
\end{defn}
\begin{rem}
An essential comment made by B. Mitchell in \cite{mitchell}
is that the class  $\Ext[C][1][][A]$ may not be a set (see \cite[Chapter 6, Exercise A]{freyd1964abelian}
for an example). Considering this fact,
we should be cautious when we talk about correspondences between extensions
classes. Nevertheless, by simplicity we will say that a correspondence
\[
\Phi:\Ext[C'][1][][A']\rightarrow\Ext[C][1][][A]
\]
is a function, if it associates to each $\overline{\eta}\in\Ext[C'][1][][A']$
a single element $\Phi(\overline{\eta})$ in $\Ext[C][1][][A]$. 
\end{rem}
Remember the following result.\\
\begin{minipage}[t]{0.58\columnwidth}%
\begin{prop}
\cite[Lemma 1.2]{maclanehomology}\label{prop:pb:operar a izquierda}
Consider a morphism $\alpha:X\rightarrow K$  and an exact sequence $\suc[][][][f][g]$
 in $\mathcal{C}$. If $(E,\alpha',g')$ is the
pullback diagram of the morphisms $g$ and $\alpha$, then there is
an exact short sequence $\eta_{\alpha}$ and a morphism $(1,\alpha',\alpha):\eta_{\alpha}\rightarrow\eta$.\end{prop}
\end{minipage}\hfill{}%
\fbox{\begin{minipage}[t]{0.37\columnwidth}%
\[
\begin{tikzpicture}[-,>=to,shorten >=1pt,auto,node distance=1cm,main node/.style=,x=1cm,y=1cm, font=\scriptsize]

   \node[main node] (C) at (0,0)      {$K$};
   \node[main node] (X0) [left of=C]  {$M$};
   \node[main node] (X1) [left of=X0]  {$N$};

   \node[main node] (X) [above of=C]  {$X$};
   \node[main node] (E) [left of=X]  {$E$};
   \node[main node] (X2) [left of=E]  {$N$};

   \node[main node] (03) [right of=C]    {$0$};
   \node[main node] (04) [left of=X1]    {$0$};

   \node[main node] (05) [right of=X]    {$0$};
   \node[main node] (06) [left of=X2]    {$0$};

\draw[->, thick]   (X0)  to node  {$g$}  (C);
\draw[->, thick]   (X1)  to node  {$$}  (X0);
\draw[->, thick]   (C)  to node  {$$}  (03);
\draw[->, thick]   (04)  to node  {$$}  (X1);

\draw[->, thick]   (X)  to node  {$\alpha$}  (C);
\draw[-, double]   (X1)  to node  {$$}  (X2);
\draw[->, thin]   (E)  to node  {$\alpha '$}  (X0);

\draw[->, thin]   (X)  to node  {$$}  (05);
\draw[->, thin]   (E)  to  node  {$g'$}  (X);
\draw[->, thin]   (X2)  to  node   {$$}  (E);
\draw[->, thin]   (06)  to  node   {$$}  (X2);

\end{tikzpicture}
\]%
\end{minipage}}

Of course the construction described above defines a correspondence
between the extension classes.
\begin{prop}
\cite[Corollary 1.2.]{mitchell}\label{cor:caracterizacion accion a izquierda}\label{cor:caracterizacion accion a derecha}
Let $\eta\in\mathcal{E}_{\mathcal{C}}(C,A)$ and $\gamma\in\Hom[\mathcal{C}][C'][C]$.
Then, the correspondence $\Phi_{\gamma}:\Ext[C][1][][A]\rightarrow\Ext[C'][1][][A]$,
$\overline{\eta}\mapsto\overline{\eta_{\gamma}}$, is a function.
\end{prop}
By duality, given a morphism $\alpha:N\rightarrow X$ and an exact
sequence
\[
\eta:\:\suc[][][][f]\mbox{,}
\]
the pushout of the morphisms $f$ and $\alpha$, gives us an exact
sequence $\eta^{\alpha}$ together with a morphism $(\alpha,\alpha',1):\eta\rightarrow\eta^{\alpha}$.
Moreover, we also have that the correspondence $\Phi^{\alpha}:\Ext[K][1][][N]\rightarrow\Ext[K][1][][X]$,
$\overline{\eta}\mapsto\overline{\eta^{\alpha}}$, is a function.
\begin{defn}
\cite[Section 1]{mitchell} For $\alpha:A\rightarrow A'$ and $\gamma:C'\rightarrow C$ morphisms in $\mathcal{C}$, 
and $E\in\Ext[C][1][][A]$, we set $E\gamma:=\Phi_{\gamma}(E)$,
and $\alpha E:=\Phi^{\alpha}(E)$.
\end{defn}
As we have described, there exists a natural action of the morphisms
on the extension classes. These actions are associative and respect
identities.
\begin{lem}
\cite[Lemma 1.3]{mitchell}\label{lem:propiedades1} Let $E\in\Ext[C][1][][A]$,
$\alpha:A\rightarrow A'$, $\alpha':A'\rightarrow A''$, $\gamma:C'\rightarrow C$,
and $\gamma':C''\rightarrow C'$ be morphisms in $\mathcal{C}$. Then, 
\begin{enumerate}
\item $1_{A}E=E$ and $E1_{C}=E$;
\item $(\alpha'\alpha)E=\alpha'(\alpha E)$ and $E(\gamma\gamma')=(E\gamma)\gamma'$; 
\item $(\alpha E)\gamma=\alpha(E\gamma)$.
\end{enumerate}
\end{lem}
Next, we recall the definition of the Baer sum.
\begin{defn}
\cite[Section 1]{mitchell} For $E,E'\in\Ext[C][1][][A]$, the sum extension
of $E$ and $E'$ is $E+E':=\nabla_{A}\left(E\oplus E'\right)\Delta_{C}\mbox{.}$
\end{defn}
This sum operation is well behaved with the actions before described
and gives a structure of abelian group to the extension classes.
\begin{thm}
\label{thm:ext1 es grupo}\label{cor:0E=00003D0}\label{lem:propiedades2}\cite[Lemma 1.4 and Theorem 1.5.]{mitchell}
For any $A,C\in\mathcal{C}$, we have that the pair $\left(\Ext[C][1][][A],+\right)$
is an abelian group, where the identity element is the extension 
given by the class of exact sequences that split. Furthermore, let $E\in\Ext[C][1][][A]$,
$E'\in\Ext[C'][1][][A']$, $\alpha\in\Hom[][A][X]$, $\alpha'\in\Hom[][A'][X']$,
$\gamma\in\Hom[][Y][C]$ and $\gamma'\in\Hom[][Y'][C']$. Then, the
following equalities hold true:
\begin{enumerate}
\item $\left(\alpha\oplus\alpha'\right)\left(E\oplus E'\right)=\alpha E\oplus\alpha'E'$;
\item $\left(\alpha+\alpha'\right)E=\alpha E+\alpha'E$;
\item $\alpha\left(E+E'\right)=\alpha E+\alpha E'$;
\item $\left(E\oplus E'\right)\left(\gamma\oplus\gamma'\right)=E\gamma\oplus E'\gamma'$;
\item $E\left(\gamma+\gamma'\right)=E\gamma+E\gamma'$; 
\item $\left(E+E'\right)\gamma=E\gamma+E'\gamma$;
\item $0E=E0=E_{0}$ for every $E\in\Ext[C][1][][A]$.
\end{enumerate}
\end{thm}

\subsection{$n$-Extensions}

We are ready for recalling the definition of $n$-extensions. It is
a well known fact that short exact sequences can be sticked together
in order to contruct a long exact sequence. Following this thought,
the spirit of $n$-extensions is to define a well behaved 1-extensions
composition that constructs long extensions.
\begin{defn}
\cite[Section 3]{mitchell} We will make use of the following considerations. 
\begin{enumerate}
\item For an exact sequence $\eta:\;\suc[A][B_{n-1}\rightarrow\cdots\rightarrow B_{0}][C]$
in $\mathcal{C}$ we say that $\eta$ is an exact sequence of length
$n$, and $A$ and $C$ are the left and right ends of $\eta$, respectively. 
\item Let $\mathcal{E}_{\mathcal{C}}^{n}(L,N)$ denote the class of exact
sequences of length $n$ with $L$ and $N$ as right and left ends.
\item Consider the following exact sequences in $\mathcal{C}$ 
\begin{alignat*}{1}
\eta:\;\suc[N][B_{n-1}\stackrel{f_{n-1}}{\rightarrow}\cdots\stackrel{f_{1}}{\rightarrow}B_{0}][K][\mu][\pi] & \mbox{,}\\
\eta':\;\suc[N'][B'_{n-1}\stackrel{f'_{n-1}}{\rightarrow}\cdots\stackrel{f'_{1}}{\rightarrow}B'_{0}][K'][\mu'][\pi'] & \mbox{.}
\end{alignat*}
 A morphism $\eta\rightarrow$$\eta'$ is a collection of $n+2$ morphisms
$(\alpha,\beta_{n-1},\cdots,\beta_{0},\gamma)$ in $\mathcal{C}$,
where $\alpha:N\rightarrow N'$, $\gamma:K\rightarrow K'$, and $\beta_{i}:B_{i}\rightarrow B'_{i}\:\forall i\in[0,n-1]$
are such that 
\[
\beta_{n-1}\mu=\mu'\alpha\mbox{, }\gamma\pi=\pi'\beta_{0}\mbox{ and }\beta_{i-1}f_{i}=f_{i}'\beta_{i}\:\forall i\in[0,n-1]\mbox{.}
\]
Equivalently, we can say that a morphism of exact sequences of length
$n$ is a commutative diagram
\end{enumerate}

\begin{minipage}[t]{1\columnwidth}%
\[
\begin{tikzpicture}[-,>=to,shorten >=1pt,auto,node distance=1cm,main node/.style=,framed, font=\scriptsize]

   \node[main node] (01) at (0,0)       {$0$};
   \node[main node] (N) [right of=01]   {$N$};
   \node[main node] (B1) [right of=N]   {$B_{n-1}$};
   \node[main node] (C) [right of=B1]   {$\cdots$};
   \node[main node] (B2) [right of=C]   {$B_0$};
   \node[main node] (K) [right of=B2]   {$K$};
   \node[main node] (02) [right of=K]   {$0$};

   \node[main node] (03) [below of=01]       {$0$};
   \node[main node] (N') [right of=03]   {$N'$};
   \node[main node] (B'1) [right of=N']   {$B'_{n-1}$};
   \node[main node] (C') [right of=B'1]   {$\cdots$};
   \node[main node] (B'2) [right of=C']   {$B'_0$};
   \node[main node] (K') [right of=B'2]   {$K'$};
   \node[main node] (04) [right of=K']   {$0$};

\draw[->, thin]   (01)  to node    {$$}    (N);
\draw[->, thin]   (N)  to node    {$$}      (B1);
\draw[->, thin]   (B1)  to node    {$$}      (C);
\draw[->, thin]   (C)  to node    {$$}      (B2);
\draw[->, thin]   (B2)  to node    {$$}    (K);
\draw[->, thin]   (K)  to node    {$$}    (02);

\draw[->, thin]   (03)  to node    {$$}    (N');
\draw[->, thin]   (N')  to node    {$$}      (B'1);
\draw[->, thin]   (B'1)  to node    {$$}      (C');
\draw[->, thin]   (C')  to node    {$$}      (B'2);
\draw[->, thin]   (B'2)  to node    {$$}    (K');
\draw[->, thin]   (K')  to node    {$$}    (04);

\draw[->, thin]   (N)  to node    {$$}      (N');
\draw[->, thin]   (B1)  to node    {$$}      (B'1);
\draw[->, thin]   (B2)  to node    {$$}    (B'2);
\draw[->, thin]   (K)  to node    {$$}    (K');

\end{tikzpicture}
\]%
\end{minipage}

\end{defn}
In the following lines, we define an equivalence relation for studying
the classes of exact sequences of length $n$. As we did for the case
with $n=1$, we start by saying that two exact sequences $\eta,\eta'\in\mathcal{E}_{\mathcal{C}}^{n}(C,A)$
are related, denoted  by $\eta\preceq\eta'$, if there is
a morphism $(1_{A},\beta_{n-1}\cdots,\beta_{0},1_{C}):\eta\rightarrow\eta'\mbox{.}$
In this case, we say also that this morphism has fixed ends. Observe
that, in contrast with the case $n=1$, this relation needs not to
be symmetric. Thus, for achieving our goal, we most consider the equivalence
relation $\equiv$ induced by $\preceq$. Namely, we write $\eta\equiv\eta'$
if there are exact sequences $\eta_{1},\cdots,\eta_{k}$ such that
\[
\eta=\eta_{1}\mbox{,}\qquad\eta_{i}\preceq\eta_{i+1}\mbox{ or }\eta_{i+1}\preceq\eta_{i},\mbox{ and }\qquad\eta'=\eta_{k}\mbox{.}
\]

\begin{defn}
\cite[Section 9]{Hilton-Stammbach} For $n\geq1$ and $A,C\in\mathcal{C}$, 
we consider the class $\Ext[C][n][][A]:=\mathcal{E}_{\mathcal{C}}^{n}(C,A)/\equiv\mbox{,}$
whose elements will be called extensions of length $n$ with $C$
and $A$ as right and left ends. Let $\overline{\eta}$ denote the
equivalence class of $\eta\in\mathcal{E}_{\mathcal{C}}^{n}(C,A)$.
An extension morphism from $\overline{\eta}$ to $\overline{\eta'}$
is just a morphism from $\eta$ to $\eta'$. \end{defn}
\begin{rem}
The definition of the equivalence relation above might seem naive.
But actually the relation is built with the purpose of making the
composition of extensions associate properly when there is a morphism
acting in the involved extensions \cite[Section 3]{mitchell}\cite[Section 5]{maclanehomology}.
In the following lines, we will discuss briefly such matter. 

Observe how in general, for $\eta\in\mathcal{E}_{\mathcal{C}}^{1}(C,A)$,
$\eta'\in\mathcal{E}_{\mathcal{C}}^{1}(D,C')$ and  $\beta:C\rightarrow C'$ in $\mathcal{C}$,
it is false that $\left(\eta\beta\right)\eta'=\eta\left(\beta\eta'\right)$.
The only affirmation that can be made is that there is an extension
morphism $\left(\eta\beta\right)\eta'\rightarrow\eta\left(\beta\eta'\right)$.
To show such morphism, we remember that $\beta$ induces morphisms $\eta\beta\rightarrow\eta\qquad$ and $\qquad\eta'\rightarrow\beta\eta'\mbox{.}$
Hence, we can build the morphisms 
\[
\left(\eta\beta\right)\eta'\rightarrow\eta\eta'\qquad\mbox{and}\qquad\eta\eta'\rightarrow\eta\left(\beta\eta'\right)\mbox{,}
\]
whose composition gives the wanted morphism. Therefore, even if we have the inequality $\left(\eta\beta\right)\eta'\neq\eta\left(\beta\eta'\right)$
we can conclude that $\overline{\left(\eta\beta\right)\eta'}=\overline{\eta\left(\beta\eta'\right)}$.\end{rem}
\begin{defn}
\cite[Section 3]{mitchell} Consider the following exact sequences of
length $n$ and $m$, respectively 
\begin{alignat*}{1}
\eta:\;\suc[N][B_{n}\rightarrow\cdots\rightarrow B_{1}][K][\mu][\pi] & \mbox{,}\\
\eta':\;\suc[K][B'_{m}\rightarrow\cdots\rightarrow B'_{1}][L][\mu'][\pi'] & .
\end{alignat*}

The composition sequence $\eta\eta'$, of $\eta$ with $\eta'$, is
the exact sequence
\[
\suc[N][B_{n}\rightarrow\cdots\rightarrow B_{1}\stackrel{\mu'\pi}{\rightarrow}B'_{m}\rightarrow\cdots\rightarrow B'_{1}][L][\mu][\pi']\mbox{.}
\]
\end{defn}
\begin{rem}
Note that each exact sequence in $\mathcal{C}$ 
\[
\kappa:\:\suc[A][B_{n}\rightarrow\cdots\rightarrow B_{1}][C]
\]
can be written as a composition of $n$ short exact sequences $\kappa=\eta_{n}\cdots\eta_{1}$,
where 
\[
\eta_{i}:=\;\suc[K_{i+1}][B_{i}][K_{i}]\mbox{,}
\]
with $K_{n+1}:=A$, $K_{1}:=C$ and $K_{i}=\im[(B_{i}\rightarrow B_{i-1})]\:\forall i\in[2,n-1]$.
We will refer to such factorization of $\kappa$ as its natural decomposition.
\end{rem}
Of course, the composition of exact sequences induces a composition
of extensions.
\begin{lem}
\cite[Proposition 3.1]{mitchell}\label{lem:composicion esta bien definido}
Let $m,n>0$, and $A,C,D\in\mathcal{C}$. Then, the correspondence
$\Phi:\Ext[C][n][][A]\times\Ext[D][m][][C]\rightarrow\Ext[D][n+m][][A]$,
$(\overline{\eta},\overline{\eta'})\mapsto\overline{\eta\eta'}$,
is a function.
\end{lem}
We can now define without ambiguity the composition of extensions.
\begin{defn}
Let $E\in\Ext[C][n][][A]$ and $E'\in\Ext[D][m][][C]$. For $E=\overline{\eta}$
and $E'=\overline{\eta'}$, we define the composition extension $EE'$
of $E$ with $E'$, as the extension $EE':=\overline{\eta\eta'}$. If $\eta=\eta_{n}\cdots\eta_{1}$
is the natural decomposition of $\eta$, the induced extension factorization
$E=\overline{\eta_{n}}\cdots\overline{\eta_{1}}$ is known as a natural
decomposition of $E$.
\end{defn}
In the same way, an $n$-extension can be factored into simpler extensions;
a morphism of $n$-extensions can be factored into a composition of
$n$ simpler morphisms. The next lemma shows the basic fact in this
matter.\\
\begin{minipage}[t]{0.6\columnwidth}%
\begin{lem}
\cite[Lemma 1.1]{mitchell}\label{lem:pb/po}\label{cor:aE=00003DE'b}
Consider a morphism of exact sequences  $(\alpha,\beta,\gamma):\eta'\rightarrow\eta$, with 
\begin{alignat*}{1}
\eta: & \quad\suc[A][B][C][f][g]\mbox{ and }\\
\eta': & \quad\suc[A'][B'][C'][f'][g']\mbox{.}
\end{alignat*}
Then, $\eta\gamma=\alpha\eta'$ and $(\alpha,\beta,\gamma)$ factors
through $\eta\gamma$ as
\[
(\alpha,\beta,\gamma)=(1,\beta',\gamma)(\alpha,\beta'',1)\mbox{.}
\]
\end{lem}
\end{minipage}\hfill{}%
\fbox{\begin{minipage}[t]{0.35\columnwidth}%
\[
\begin{tikzpicture}[-,>=to,shorten >=1pt,auto,node distance=1cm,main node/.style=, font=\scriptsize]

   \node[main node] (I1)   at (0,0)            {$0$};
   \node[main node] (A1)  [right of=I1]        {$A'$};
   \node[main node] (B1)  [right of=A1]         {$B'$};
   \node[main node] (C1)  [right of=B1]         {$C'$};
   \node[main node] (D1) [right of=C1]         {$0$};

   \node[main node] (I2)  [below of=I1]         {$0$};
   \node[main node] (A2)  [below of=A1]         {$A$};
   \node[main node] (B2)  [below of=B1]         {$E$};
   \node[main node] (C2)  [below of=C1]         {$C'$};
   \node[main node] (D2)  [below of=D1]         {$0$};

   \node[main node] (I3)  [below of=I2]         {$0$};
   \node[main node] (A3)  [below of=A2]         {$A$};
   \node[main node] (B3)  [below of=B2]         {$B$};
   \node[main node] (C3)  [below of=C2]         {$C$};
   \node[main node] (D3)  [below of=D2]         {$0$};

\draw[->, thin]   (I1)  to  node     {$$}         (A1);
\draw[->, thin]   (A1)  to  node     {$f'$}         (B1);
\draw[->, thin]   (B1)  to  node     {$g'$}         (C1);
\draw[->, thin]   (C1)  to  node     {$$}         (D1);

\draw[->, thin]   (I2)  to  node     {$$}         (A2);
\draw[->, thin]   (A2)  to  node     {$$}         (B2);
\draw[->, thin]   (B2)  to  node     {$$}         (C2);
\draw[->, thin]   (C2)  to  node     {$$}         (D2);

\draw[->, thin]   (I3)  to  node     {$$}         (A3);
\draw[->, thin]   (A3)  to  node     {$f$}         (B3);
\draw[->, thin]   (B3)  to  node     {$g$}         (C3);
\draw[->, thin]   (C3)  to  node     {$$}         (D3);

\draw[->, thin]   (A1)  to  node     {$\alpha$}         (A2);
\draw[->, thin]   (B1)  to  node     {$\beta ''$}         (B2);
\draw[-, double]   (C1)  to  node     {$$}         (C2);

\draw[-, double]   (A2)  to  node     {$$}         (A3);
\draw[->, thin]   (B2)  to  node     {$\beta '$}         (B3);
\draw[->, thin]   (C2)  to  node     {$\gamma $}         (C3);

\end{tikzpicture}
\]%
\end{minipage}}

In general, we can make the following affirmation.
\begin{cor}
\label{cor:asociatividad con morfismos}Let $\eta,\eta'\in\mathcal{E}_{\mathcal{C}}^{n}(C,A)$
be exact sequences with natural decompositions $\eta=\eta_{n}\cdots\eta_{1}$
and $\eta'=\eta'_{n}\cdots\eta'_{1}$. Then, the following statements
hold true.
\begin{enumerate}
\item There is an exact sequence morphism $(\alpha,\beta_{n-1},\cdots,\beta_{0},\gamma):\eta\rightarrow\eta'$
if, and only if, there is a collection of extension morphisms 
\[
(\alpha_{i},\beta_{i-1},\alpha_{i-1}):\overline{\eta{}_{i}}\rightarrow\overline{\eta'_{i}}\:\forall i\in[1,n]
\]
where $\alpha_{n}=\alpha$ and $\alpha_{0}=\gamma$.
\item If there is an exact sequence morphism $(\alpha,\beta_{n-1},\cdots,\beta_{0},\gamma):\eta\rightarrow\eta'$,
then there is a collection of morphisms $\alpha_{n-1},\cdots,\alpha_{1}$
in $\mathcal{C}$ satisfying the following equalities:

\begin{enumerate}
\item $\overline{\eta'_{n}}\cdots\overline{\eta'_{1}}\gamma=\alpha\overline{\eta_{n}}\cdots\overline{\eta_{1}}$,
\item $\overline{\eta'_{i}}\cdots\overline{\eta'_{1}}\gamma=\alpha_{i}\overline{\eta_{i}}\cdots\overline{\eta_{1}}\:\forall i\in[1,n-1]$, and
\item $\overline{\eta'_{n}}\cdots\overline{\eta'_{i+1}}\alpha_{i}=\alpha\overline{\eta_{n}}\cdots\overline{\eta_{i+1}}\:\forall i\in[1,n-1]$. 
\end{enumerate}
\end{enumerate}
\end{cor}
\begin{proof}
It follows from \ref{lem:pb/po}.
\end{proof}
By Lemma \ref{lem:composicion esta bien definido}, the following
actions are well defined.
\begin{defn}
\cite[Section 3]{mitchell} Consider $\eta,\eta'\in\mathcal{E}_{\mathcal{C}}^{n}(C,A)$,
$E:=\overline{\eta}\in\Ext[C][n][][A]$, $E':=\overline{\eta'}\in\Ext[C][n][][A]$, and let 
$\eta=\eta_{n}\cdots\eta_{1}$ and $\eta'=\eta'_{n}\cdots\eta'_{1}$
be the natural decompositions of $\eta$ and $\eta'$.
\begin{enumerate}
\item Given $\alpha\in\Hom[][A][A']$, we define $\alpha E:=\alpha\overline{\eta_{n}}\cdots\overline{\eta_{1}}\mbox{.}$
\item Given $\gamma\in\Hom[][C'][C]$, we define $E\gamma:=\overline{\eta_{n}}\cdots\overline{\eta_{1}}\gamma\mbox{.}$
\item We define the sum of extensions of length $n$ in the following way
\[
E+E':=\nabla_{A}\left(E\oplus E'\right)\Delta_{C}\mbox{.}
\]

\end{enumerate}
\end{defn}
Most of the properties, proved earlier for extensions of length 1,
can be naturally extended, as can be seen in the following lines.
\begin{cor}
\cite[Lemma 3.2 an Theorem 3.3]{mitchell}\label{cor:props comps extn}\label{thm:extn es grupo}
Let $n>0$.
\begin{enumerate}
\item Let  $E\in\Ext[C][n][][A]$, $E'\in\Ext[D][m][][C']$, $\beta\in\Hom[][C'][C]$,
$\beta'\in\Hom[][C''][C']$, $\alpha\in\Hom[][A][A']$, and $\alpha'\in\Hom[][A'][A'']$.
Then the following equalities hold true: 

\begin{enumerate}
\item $\left(E\beta\right)E'=E\left(\beta E'\right)$;
\item $1_{A}E=E=E1_{C}$;
\item $E\left(\beta\beta'\right)=\left(E\beta\right)\beta'$; 
\item $\left(\alpha'\alpha\right)E=\alpha'\left(\alpha E\right)$. 
\end{enumerate}
\item Let  $E\in\Ext[C][n][][A]$, $E'\in\Ext[C'][n][][A']$, $F\in\Ext[D][m][][C]$, 
$F'\in\Ext[D'][m][][C']$,   $\alpha\in\Hom[][A][X]$, $\alpha'\in\Hom[][A'][X']$, $\gamma\in\Hom[][Y][C]$,
and $\gamma'\in\Hom[][Y'][C']$. Then the following equalities hold
true: 

\begin{enumerate}
\item $(\alpha\oplus\alpha')\left(E\oplus E'\right)=\alpha E\oplus\alpha'E'$
and $\left(E\oplus E'\right)\left(\gamma\oplus\gamma'\right)=E\gamma\oplus E'\gamma'$;
\item $\left(E\oplus E'\right)\left(F\oplus F'\right)=EF\oplus E'F'$;
\item $\left(E+E'\right)F=EF+E'F$ and $E\left(F+F'\right)=EF+EF'$;
\item $(\alpha+\alpha')E=\alpha E+\alpha'E$ and $E\left(\gamma+\gamma'\right)=E\gamma+E\gamma'$;
and
\item $\alpha\left(E+E'\right)=\alpha E+\alpha E'$ and $\left(E+E'\right)\gamma=E\gamma+E'\gamma$.
\end{enumerate}
\item The pair $\left(\Ext[C][n][][A],+\right)$ is an abelian group, where
the identity element is the extension $E_{0}$ given by the exact
sequence, in case $n\geq2$, 
\[
\suc[A][A\stackrel{0}{\rightarrow} \cdots \stackrel{0}{\rightarrow} C][C][1][1]  \mbox{ .} 
\]

\end{enumerate}
\end{cor}
We conclude this section with the following theorem that focus on
characterizing the trivial extensions.
\begin{thm}
\cite[Theorem 4.2]{mitchell}\label{thm:E=00003D0} Let $n>1$ and $\eta\in\mathcal{E}_{\mathcal{C}}^{n}(C,A)$
with a natural decomposition $\eta=\eta_{n}\cdots\eta_{1}$. Then
, the following statements hold true:
\begin{enumerate}
\item $\overline{\eta}=0$; 
\item there is an exact sequence $\kappa$$\in\mathcal{E}_{\mathcal{C}}^{n}(C,A)$
and a pair of morphisms with fixed ends $0\leftarrow\kappa\rightarrow\eta\mbox{.}$
\item there is an exact sequence $\kappa'\in\mathcal{E}_{\mathcal{C}}^{n}(C,A)$
and a pair of morphisms with fixed ends $0\rightarrow\kappa'\leftarrow\eta\mbox{.}$

\end{enumerate}
\end{thm}

\section{Additional structure in Abelian Categories}

In this section we will approach our problem dealing with arbitrary
products and coproducts. Of course, an abelian category does not necessarily
have arbitrary products and coproducts. Hence, we will review briefly
the theory of abelian categories with additional structure introduced
by A. Grothendieck in \cite{Ab}. For further reading we suggest \cite[Section 2.8]{Popescu}.

\subsection{Limits and colimits}
$ $
\begin{defn}
\cite[Section 1.4.]{Popescu} Let $\mathcal{C}$ and $I$ be categories, where
$I$ is small (that is the class of objects of $I$ is a set). Let
$F:I\rightarrow\mathcal{C}$ be a functor and $X\in\mathcal{C}$.
A family of morphisms $\left\{ \alpha_{i}:F(i)\rightarrow X\right\} _{i\in I}$
in $\mathcal{C}$ is co-compatible with $F$, if $\alpha_{i}=\alpha_{j}F(\lambda)$
for every $\lambda:i\rightarrow j$ in $I$.\\
\begin{minipage}[t]{0.7\columnwidth}%
The colimit (or inductive limit) of $F$ is an object $\colim$ in $\mathcal{C}$
with a co-compatible family of morphisms 
\[
\left\{ \mu_{i}:F(i)\rightarrow\colim\right\} _{i\in I}\mbox{,}
\]
such that for every co-compatible family of morphisms $\left\{ \gamma_{i}:F(i)\rightarrow X\right\} _{i\in I}$,
there is a unique morphism $\gamma:\colim\rightarrow X$ such that
$\gamma_{i}=\gamma\mu_{i}$ for every $i\in I$.
\end{minipage}\hfill{}%
\fbox{\begin{minipage}[t]{0.25\columnwidth}%
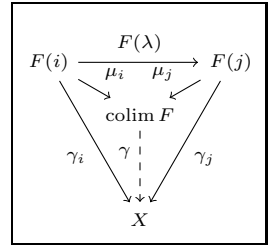
\[
\begin{tikzpicture}[-,>=to,shorten >=1pt,auto,node distance=3cm,main node/.style=, font=\scriptsize]

\coordinate (A) at (150:1.4cm);
\coordinate (B) at (30:1.4cm);
\coordinate (C) at (270:1.4cm);

\node (Fi) at (barycentric cs:A=1,C=0,B=0 ) {$F(i) $};
\node (Fj) at (barycentric cs:A=0,C=0,B=1 ) {$F(j) $};
\node (X) at (barycentric cs:A=0,C=1,B=0 ) {$X $};
\node (L) at (barycentric cs:A=1,C=1,B=1 ) {$\operatorname{colim}F $};

\draw[->, thin]  (Fi)  to  node  {$F( \lambda )$} (Fj);
\draw[->, thin]  (Fi)  to  node [below left] {$\gamma _i$} (X);
\draw[->, thin]  (Fj)  to  node [below right] {$\gamma _j$} (X);
\draw[->, thin]  (Fi)  to  node  {$\mu _i$} (L);
\draw[->, thin]  (Fj)  to [above left] node  {$\mu _j$} (L);
\draw[->, dashed]  (L)  to [above left]  node  {$\gamma$} (X);

\end{tikzpicture}
\]%
\end{minipage}}
\end{defn}

Let $I$ be a small category and $\lambda:i\rightarrow j$ be a morphism
in $I$. The following notation will be useful $s(\lambda):=i$ and
$t(\lambda):=j$.
\begin{prop}
\label{prop:construccion colimites}\cite[Proposition 8.4]{ringsofQuotients}
Let $\mathcal{C}$ be a preadditive category with coproducts and cokernels,
$I$ be a small category, $F:I\rightarrow\mathcal{C}$ be a functor,
and
\[
u_{k}:F(k)\rightarrow\bigoplus_{i\in I}F(i)\quad\forall k\in I\mbox{,}\quad v_{\lambda}:F(s(\lambda))\rightarrow\bigoplus_{\gamma\in H}F(s(\gamma))\quad\forall\lambda\in H:=\mbox{Hom}_{I}
\]
be the respective canonical inclusions into the coproducts. Then,\\
\begin{minipage}[t]{0.48\columnwidth}%
{\footnotesize
\[
 \colim=\Cok[\left(\bigoplus_{\gamma\in H}F(s(\gamma))\overset{\varphi}{\rightarrow}\bigoplus_{i\in I}F(i)\right)]\mbox{,}
\]}
where $\varphi$ is the morphism induced by the universal property
of coproducts applied to the family of morphisms
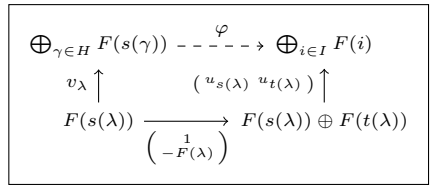
\[
\left\{ \varphi_{\lambda}:=u_{s(\lambda)}-u_{t(\lambda)}F(\lambda)\right\} _{\lambda\in H}\mbox{.}
\]
\end{minipage}\hfill{}%
\begin{minipage}[t]{0.49\columnwidth}%
\[
\begin{tikzpicture}[-,>=to,shorten >=1pt,auto,node distance=1cm,main node/.style=,framed, font=\scriptsize]

   \node[main node] (1) at (0,0)            {$\bigoplus_{\gamma\in H}F(s(\gamma))$};
   \node[main node] (2) at (3,0)        {$\bigoplus_{i\in I}F(i)$};
   \node[main node] (3) [below of=1]        {$F(s( \lambda))$};
   \node[main node] (4) [below of=2]       {$F(s( \lambda)) \oplus F(t( \lambda))$};

\draw[->, dashed]   (1)  to node    {$\varphi$}    (2);
\draw[->, thin]   (3)  to [below] node    {$\left(\begin{smallmatrix}1\\-F(\lambda)\end{smallmatrix}\right)$}    (4);
\draw[->, thin]   (3)  to node    {$v_{\lambda}$}    (1);
\draw[->, thin]   (4)  to [left] node    {$\left(\begin{smallmatrix}u_{s(\lambda)} & u_{t(\lambda)}\end{smallmatrix}\right)$}    (2);

\end{tikzpicture}
\]%
\end{minipage}
\end{prop}
The dual notion of colimit is the limit.
\begin{defn}
\cite[Section 1.4.]{Popescu} Let $\mathcal{C}$ and $I$ be categories, with
$I$ small. Let $F:I\rightarrow\mathcal{C}$ be a functor and $X\in\mathcal{C}$.
A family of morphisms $\left\{ \alpha_{i}:X\rightarrow F(i)\right\} _{i\in I}$
in $\mathcal{C}$ is compatible with $F$, if $\alpha_{j}=F(\lambda)\alpha_{i}$
for every $\lambda:i\rightarrow j$ in $I$.\\
\begin{minipage}[t]{0.7\columnwidth}%
The limit (or projective limit) of $F$ is an object $\lim F$ in $\mathcal{C}$
together with a compatible family of morphisms 
\[
\left\{ \mu_{i}:\lim F\rightarrow F(i)\right\} _{i\in I}
\]
such that for any compatible family of morphisms $\left\{ \gamma_{i}:X\rightarrow F(i)\right\} _{i\in I}$
there is a unique  $\gamma \in \operatorname{Hom}_{\mathcal{C}}(X, \lim F )$ such that
$\gamma_{i}=\mu_{i}\gamma$ for every $i\in I$.
\end{minipage}\hfill{}%
\fbox{\begin{minipage}[t]{0.25\columnwidth}%
\[
\begin{tikzpicture}[-,>=to,shorten >=1pt,auto,node distance=3cm,main node/.style=, font=\scriptsize]

\coordinate (A) at (150:1.4cm);
\coordinate (B) at (30:1.4cm);
\coordinate (C) at (270:1.4cm);

\node (Fi) at (barycentric cs:A=1,C=0,B=0 ) {$F(i) $};
\node (Fj) at (barycentric cs:A=0,C=0,B=1 ) {$F(j) $};
\node (X) at (barycentric cs:A=0,C=1,B=0 ) {$X $};
\node (L) at (barycentric cs:A=1,C=1,B=1 ) {$\operatorname{lim}F $};

\draw[->, thin]  (Fi)  to  node  {$F( \lambda )$} (Fj);
\draw[->, thin]  (X)  to  node [below left] {$\gamma _i$} (Fi);
\draw[->, thin]  (X)  to  node [below right] {$\gamma _j$} (Fj);
\draw[->, thin]  (L)  to  node [above right] {$\mu _i$} (Fi);
\draw[->, thin]  (L)  to [above left] node  {$\mu _j$} (Fj);
\draw[->, dashed]  (X)  to [above left] node  {$\gamma$} (L);

\end{tikzpicture}
\]%
\end{minipage}}
\end{defn}
\begin{prop}
\label{prop:construccion limites}\cite[Proposition 8.2]{ringsofQuotients}
Let $\mathcal{C}$ be a preadditive category with products and kernels,
$I$ be an small category, $F:I\rightarrow\mathcal{C}$ be a functor,
and 
\[
u_{k}:\prod_{i\in I}F(i)\rightarrow F(k)\quad\forall k\in I,\quad v_{\lambda}:\prod_{\gamma\in H}F(t(\gamma))\rightarrow F(t(\lambda))\quad\forall\lambda\in H:=\mbox{Hom}_{I}
\]
be the respective canonical proyections out of the products. Then,\\
\begin{minipage}[t]{0.46\columnwidth}%
{\small
\[
\lim F=\Ker[\left(\prod_{i\in I}F(i)\overset{\varphi}{\rightarrow}\prod_{\gamma\in H}F(t(\gamma))\right)]\mbox{,}
\]}
where $\varphi$ is the morphism induced by the universal property
of products applied to the family of morphisms
\[
\left\{ \varphi_{\lambda}:=F(\lambda)u_{s(\lambda)}-u_{t(\lambda)}\right\} _{\lambda\in H}\mbox{.}
\]
\end{minipage}\hfill{}%
\begin{minipage}[t]{0.49\columnwidth}%
\[
\begin{tikzpicture}[-,>=to,shorten >=1pt,auto,node distance=1.5cm,main node/.style=,framed]

   \node[main node] (1) at (0,0)            {$\prod_{i\in I}F(i)$};    \node[main node] (2) at (3,0)        {$\prod_{\gamma\in H}F(t(\gamma))$};    \node[main node] (3) [below of=1]        {$F(s( \lambda)) \oplus F(t( \lambda))$};    \node[main node] (4) [below of=2]       {$F(t( \lambda))$};
\draw[->, dashed]   (1)  to node    {$\varphi$}    (2); \draw[->, thin]   (1)  to [left] node    {$\left(\begin{smallmatrix}u_{s(\lambda)}\\u_{t(\lambda)}\end{smallmatrix}\right)$}    (3); \draw[->, thin]   (2)  to node    {$v_{\lambda}$}    (4); \draw[->, thin]   (3)  to [below] node    {$\left(\begin{smallmatrix} -F(\lambda) & 1\end{smallmatrix}\right)$}    (4);
   
   \end{tikzpicture}
\]%
\end{minipage}\end{prop}
\begin{defn}
Let $I$ be a small category and $\mathcal{C}$ be an abelian category.
It is said that a family of objects and morphisms $\left(M_{i},f_{\alpha}\right)_{i\in I,\alpha\in\mbox{Hom}_{I}}$
is a direct system, if there is a functor $F:I\rightarrow\mathcal{C}$
such that $F(i)=M_{i}\,\forall i\in I$ and $F(\alpha)=f_{\alpha}$ for every $\alpha\in\mbox{Hom}_{I}$.
\end{defn}

\subsection{Ab3 and Ab4 Categories}
\begin{defn}
\cite[Section 2.8.]{Popescu} An Ab3 category is an abelian category satisfying
the following condition:
\begin{description}
\item [{(Ab3)}] For every set of objects $\left\{ A_{i}\right\} _{i\in I}$
in $\mathcal{C}$, the coproduct $\bigoplus_{i\in I}A_{i}$ exists.
\end{description}
\end{defn}
We remember the following well known fact.
\begin{prop}
\cite[Section 2.8.]{Popescu} Let $\mathcal{C}$ be an Ab3 category and 
\[
\left\{ X'_{i}\overset{f_{i}}{\rightarrow}X_{i}\overset{g_{i}}{\rightarrow}X''_{i}\rightarrow0\right\} _{i\in I}
\]
be a set of exact sequences in $\mathcal{C}$. Then, 
\[
\bigoplus_{i\in I}X'_{i}\overset{\bigoplus_{i\in I}f_{i}}{\rightarrow}\bigoplus_{i\in I}X_{i}\overset{\bigoplus_{i\in I}g_{i}}{\rightarrow}\bigoplus_{i\in I}X''_{i}\rightarrow0
\]
is an exact sequence in $\mathcal{C}$.
\end{prop}
In general, it is not possible to prove that $\bigoplus_{i\in I}f_{i}$
is a monomorphism if each $f_{i}$ is a monomorphism. For this reason, 
the following Grothendieck's condition arised. 
\begin{defn}
\cite[Proposition 8.3.]{Popescu} An Ab4 category is an Ab3 category $\mathcal{C}$
satisfying the following condition:
\begin{description}
\item [{(Ab4)}] for every set of monomorphisms $\left\{ f_{i}:X_{i}\rightarrow Y_{i}\right\} _{i\in I}$
in $\mathcal{C}$, the morphism $\bigoplus_{i\in I}f_{i}$ is a monomorphism.
\end{description}
\end{defn}
We will refer to the dual condition as Ab4{*}.
\begin{rem}
Let $\mathcal{C}$ be an Ab4 category. Then, for every sets of objects
$\{A_{i}\}_{i\in I}$ and $\{B_{i}\}_{i\in I}$ in $\mathcal{C}$,
the correspondence 
\begin{align*}
C:\prod_{i\in I}\Ext[A_{i}][n][][B_{i}] & \rightarrow\Ext[\bigoplus_{i\in I}A_{i}][n][][\bigoplus_{i\in I}B_{i}]\mbox{, }(\overline{\eta_{i}})\mapsto\overline{\bigoplus_{i\in I}\eta_{i}}\mbox{,}
\end{align*}
is a well defined morphism of abelian groups.
\end{rem}

\subsection{Ext groups and arbitrary products and coproducts}

We are finally ready to proceed in our goal's direction.
\begin{lem}
\label{lem:n-pushouy de monos} Let $\mathcal{C}$ be an Ab4 category,
and 
\[
\left\{ \eta_{i}:\quad\suc[B][A_{i}][C_{i}][f_{i}][g_{i}]\right\} _{i\in I}
\]
be a set of short exact sequences in $\mathcal{C}$. Then, there is
a short exact sequence 
\[
\eta:\quad\suc[B][\mbox{colim}(f_{i})][\bigoplus_{i\in I}C_{i}][f][g]
\]
such that $\eta\mu_{i}=\eta_{i}\:\forall i\in I$, where $\left\{ \mu_{i}:C_{i}\rightarrow\bigoplus_{i\in I}C_{i}\right\} _{i\in I}$
is the family of canonical inyections into the coproduct. \end{lem}
\begin{proof}
Consider the set $\left\{ f_{i}:B\rightarrow A_{i}\right\} _{i\in I}$
as a direct system. Observe that the set of morphisms of exact sequences
\[
\left\{ (1_{B},f_{i},0):\beta\rightarrow\eta_{i}\right\} _{i\in I}\quad\mbox{with }\quad\beta:=\quad\suc[B][B][0][1][0]\mbox{,}
\]
is a direct system of exact sequences. We will consider the colimit
of such system and prove that, as result, we get a short exact sequence.
To this end, we observe that $(B,1_{i}:B\rightarrow B)_{i\in I}$
is the colimit of the system $\left\{ 1_{i}:B\rightarrow B\right\} _{i\in I}$
and that $(\bigoplus_{i\in I}C_{i},\mu_{i}:C_{i}\rightarrow\bigoplus_{i\in I}C_{i})$
is the colimit of the system $\left\{ 0:0\rightarrow C_{i}\right\} _{i\in I}$.\\
\begin{minipage}[t]{0.33\columnwidth}%
Hence, by \ref{prop:construccion colimites} we build the diagram
beside, where the columns are the morphism mentioned in \ref{prop:construccion colimites},
the upper and central rows are coproducts of the sequences $\beta$
and $\eta_{i}$ respectively, and the bottom row is the result of
the colimits. Thus, by the snake lemma we get the exact sequence %
\end{minipage}\hfill{}%
\fbox{\begin{minipage}[t]{0.62\columnwidth}%
\[
\begin{tikzpicture}[-,>=to,shorten >=1pt,auto,node distance=2.5cm,main node/.style=,x=2cm,y=.9cm, font=\scriptsize]

   \node[main node] (01) at (3,1)            {$\scriptstyle 0$};
   \node[main node] (02) at (0.3,0)            {$\scriptstyle 0$};
   \node[main node] (03) at (3.6,0)            {$\scriptstyle 0$};
   \node[main node] (04) at (0.3,-1)            {$\scriptstyle 0$};
   \node[main node] (05) at (3.6,-1)            {$\scriptstyle 0$};
   \node[main node] (06) at (1,-3)            {$\scriptstyle 0$};
   \node[main node] (07) at (2,-3)            {$\scriptstyle 0$};
   \node[main node] (08) at (3,-3)            {$\scriptstyle 0$};
   \node[main node] (09) at (3.6,-2)            {$\scriptstyle 0$};
   \node[main node] (X) at (2,-.5)            {$\scriptstyle $};

   \node[main node] (1) at (1,0)            {$\scriptstyle \bigoplus _{i\in I} B_i$};
   \node[main node] (2) at (2,0)            {$\scriptstyle \bigoplus _{i\in I} B_i$};
   \node[main node] (3) at (3,0)            {$\scriptstyle 0$};
   \node[main node] (4) at (1,-1)            {$\scriptstyle B \oplus ( \bigoplus _{i\in I} B_i )$};
   \node[main node] (5) at (2,-1)            {$\scriptstyle B \oplus ( \bigoplus _{i\in I} A_i )$};
   \node[main node] (6) at (3,-1)            {$\scriptstyle \bigoplus _{i\in I} C_i$};
   \node[main node] (7) at (1,-2)            {$\scriptstyle B$};
   \node[main node] (8) at (2,-2)            {$\scriptstyle \operatorname{colim}(f_i)$};
   \node[main node] (9) at (3,-2)            {$\scriptstyle \bigoplus _{i\in I} C_i$};

\draw[->, thin]   (01)  to node    {$$}    (3);
\draw[->, thin]   (02)  to node    {$$}    (1);
\draw[->, thin]   (1)  to node    {$$}    (2);
\draw[->, thin]   (2)  to node    {$$}    (3);
\draw[->, thin]   (3)  to node    {$$}    (03);

\draw[->, thin]   (04)  to node    {$$}    (4);
\draw[->, thin]   (4)  to node    {$$}    (5);
\draw[->, thin]   (5)  to node    {$$}    (6);
\draw[->, thin]   (6)  to node    {$$}    (05);

\draw[->, thin]   (7)  to node    {$$}    (8);
\draw[->, thin]   (8)  to node    {$$}    (9);
\draw[->, thin]   (9)  to node    {$$}    (09);

\draw[->, thin]   (1)  to node    {$$}    (4);
\draw[->, thin]   (4)  to node    {$$}    (7);
\draw[->, thin]   (7)  to node    {$$}    (06);

\draw[->, thin]   (2)  to node    {$$}    (5);
\draw[->, thin]   (5)  to node    {$$}    (8);
\draw[->, thin]   (8)  to node    {$$}    (07);

\draw[->, thin]   (01)  to node    {$$}    (3);
\draw[->, thin]   (3)  to node    {$$}    (6);
\draw[->, thin]   (6)  to node    {$$}    (9);
\draw[->, thin]   (9)  to node    {$$}    (08);

\draw[-, thin]   (01) ..controls (4,-.3).. (X);
\draw[->, thin]   (X) ..controls (0,-.7)..   (7);

\draw[-, thin]   (01) ..controls (4,-.3).. (2,-.5);
\draw[->, thin]   (2,-.5) ..controls (0,-.7)..   (7);

\end{tikzpicture}
\]%
\end{minipage}}\\
 
\[
\eta:\quad\suc[B][\mbox{colim}(f_{i})][\bigoplus_{i\in I}C_{i}]\mbox{.}
\]
Furthermore, the families of morphisms associated to such colimits
give us the exact sequence morphisms $(1,\mu'_{i},\mu_{i}):\eta_{i}\rightarrow\eta\;\forall i\in I\mbox{,}$
which proves the statement. \end{proof}
\begin{prop}
\label{prop:ext1 y coprods arb}Let $\mathcal{C}$ be an Ab4 category and
$\{A_{i}\}_{i\in I}$ a set of objects in $\mathcal{C}$.
Consider the coproduct with the canonical
inclusions $\left(\mu_{i}:A_{i}\rightarrow\bigoplus_{i\in I}A_{i}\right)_{i\in I}$.
Then, the correspondence $\Psi:\Ext[\bigoplus A_{i}][1][][B]\rightarrow\prod_{i\in I}\Ext[A_{i}][1][][B]$, defined by
$E\mapsto\left(E\mu_{i}\right)_{i\in I}$, is an isomorphism for every $B\in \mathcal{C}$. \end{prop}
\begin{proof}
We will proceed by proving the following steps:
\begin{enumerate}
\item The correspondence $\Psi$ is an abelian group morphism.
\item $\Psi$ is injective.
\item Given $(\overline{\eta_{i}})\in\prod_{i\in I}\Ext[A_{i}][1][][B]$, there is
$E\in\Ext[\bigoplus A_{i}][1][][B]$ such that $\Psi(E)=(\overline{\eta_{i}})$.
\end{enumerate}

Clearly, proving these statements are enough to conclude the desired
proposition.
\begin{enumerate}
\item It follows by \ref{cor:asociatividad con morfismos}.
\item Suppose that $E$ is an extension with representative 
\[
\eta:\;\suc[B][C][\bigoplus_{i\in I}A_{i}][f][g]
\]
such that $E\mu_{i}=0$  $\forall i\in I$. Suppose that $(1,p_{i},\mu_{i}):E\mu_{i}\rightarrow E$
is the morphism induced by $\mu_{i}$, and that each extension $E\mu_{i}$
has as representative the exact \\
\begin{minipage}[t]{0.45\columnwidth}%
sequence $\eta_{i}:\;\suc[B][C_{i}][A_{i}][f_{i}][g_{i}]\mbox{.}$
By definition, there is a morphism $h_{i}:A_{i}\rightarrow C_{i}$
such that $g_{i}h_{i}=1_{A_{i}}$. Thus, by the coproduct universal
property, there is a unique morphism $h:\bigoplus_{i\in I}A_{i}\rightarrow C$
such that $h\mu_{i}=p_{i}h_{i}$ $\forall i\in\{1,2\}$. Therefore,
by 
\[
gh\mu_{i}=gp_{i}h_{i}=\mu_{i}g_{i}h_{i}=\mu_{i}\:\forall i\in I,
\]
\end{minipage}\hfill{}%
\fbox{\begin{minipage}[t]{0.45\columnwidth}%
\[
\begin{tikzpicture}[-,>=to,shorten >=1pt,auto,node distance=1.3cm,main node/.style=,x=1.8cm,y=1.5cm, font=\scriptsize]

   \node[main node] (1) at (0,0)        {$0$};
   \node[main node] (2) at (0.5,0)    {$B$};
   \node[main node] (3) [right of=2]    {$C_i$};
   \node[main node] (4) [right of=3]    {$A_i$};
   \node[main node] (5) at (2.6,0)    {$0$};
   \node[main node] (6) [below of=1]    {$0$};
   \node[main node] (7) [below of=2]    {$B$};
   \node[main node] (8) [right of=7]    {$C$};
   \node[main node] (9) [right of=8]    {$\bigoplus _{i\in I} A_i$};
   \node[main node] (0) [below of=5]    {$0$};

\draw[->, thin]   (1)  to node  [above]{$$}     (2);
\draw[->, thin]   (2)  to node  [below]{$f_i$}  (3);
\draw[->, thin]   (3)  to node  [below]{$g_i$}  (4);
\draw[->, thin]   (4)  to node  [above]{$$}     (5);

\draw[->, thin]   (6)  to node  [below]{$$}     (7);
\draw[->, thin]   (7)  to node  [above]{$f$}  (8);
\draw[->, thin]   (8)  to node  [above]{$g$}  (9);
\draw[->, thin]   (9)  to node  [below]{$$}     (0);

\draw[-, double]   (2)  to node  [above]{$$}    (7);
\draw[->, thin]   (3)  to node  {$p_i$}    (8);
\draw[->, thin]   (4)  to node  {$\mu _i$}     (9);

\draw[->, thin]   (4)    to [bend right=30] node   [above]{$h_i$}  (3);
\draw[->, dashed]   (9)  to [bend left=30]  node   [below]{$h$}    (8);
   
\end{tikzpicture}
\]%
\end{minipage}}\\
 we have that $gh=1_{\bigoplus_{i\in I}A_{i}}$ by the coproduct universal
property; and thus, $E=0$.
\item It follows by \ref{lem:n-pushouy de monos}.
\end{enumerate}
\end{proof}
\begin{thm}
\label{prop:extn coprods arb} Let $\mathcal{C}$ be an Ab4 category,
$n\geq1$, and $\{A_{i}\}_{i\in I}$ be a set of objects in $\mathcal{C}$.
 Consider the coproduct $\bigoplus _{i \in I} A_i$ and the canonic inclusions  $\left(\mu_{i}:A_{i}\rightarrow\bigoplus_{i\in I}A_{i}\right)_{i\in I}$.
Then, the correspondence $\Psi_{n}:\Ext[\bigoplus A_{i}][n][][B]\rightarrow\prod_{i\in I}\Ext[A_{i}][n][][B]$,
$E\mapsto\left(E\mu_{i}\right)_{i\in I}$, is an isomorphism of abelian
groups for every $B\in \mathcal{C}$.\end{thm}
\begin{proof}
We will proceed by proving the following statements:
\begin{enumerate}
\item The correspondence $\Psi_{n}$ is a morphism of abelian groups;
\item $\Psi_{n}$ is injective;
\item For every  $(\overline{\eta_{i}})\in\prod_{i\in I}\Ext[A_{i}][n][][B]$,
there is  $E\in \Ext[\bigoplus _{i \in I} A_{i}][n][][B]$ such that $\Psi_{n}(E)=(\overline{\eta_{i}})$.
\end{enumerate}
It is worth to mention that the result was already proved in \ref{prop:ext1 y coprods arb}
for $n=1$. Furthermore, in the proof of \ref{prop:ext1 y coprods arb}(c)
it was shown explicitly the inverse function of $\Psi_{1}$. We will
denote such correspondence as $\Psi_{1}^{-1}$.
\begin{enumerate}
\item It follows by \ref{cor:asociatividad con morfismos}.
\item Let $\overline{\eta}$ be an extension with a natural decomposition
$\overline{\eta}=\overline{\eta_{n}}\cdots\overline{\eta_{1}}$ such
that $\overline{\eta}\mu_{i}=0\:\forall i\in I$. By \ref{thm:E=00003D0}
this means that for every $i\in I$ there is a pair of exact sequences
morphisms with fixed ends $\eta\mu_{i}\leftarrow\kappa_{i}\rightarrow0$.
Suppose that each exact sequence $\kappa_{i}$ has the natural decomposition $\kappa_{i}=\kappa(i)_{n}\cdots\kappa(i)_{1}\mbox{.}$ It follows from the morphism $\kappa_{i}\rightarrow0$ that
\[
\kappa(i)_{n}:=\kappa'_{i}:\quad\suc[B][Y_{i}][X_{i}][f_{i}][g_{i}]
\]
is a splitting exact sequence. Let $(\overline{\kappa'_{i}}):=(\overline{\kappa'_{i}})_{i\in I}\in\prod_{i\in I}\Ext[X_{i}][1][][B]\mbox{.}$  By \ref{prop:ext1 y coprods arb}(c), we know that  $\Psi_{1}^{-1}(\overline{\kappa'_{i}})\in\Ext[\bigoplus_{i\in I}X_{i}][1][][B]$
is an extension such that $\Psi_{1}^{-1}(\overline{\kappa'_{i}})\mu'_{i}=\kappa'_{i}\forall i\in I$,
where each $\mu'_{i}:X_{i}\rightarrow\bigoplus_{i\in I}X_{i}$ is
 the canonical inclusion. Let $\kappa:=\Psi_{1}^{-1}(\overline{\kappa'_{i}})\left(\overline{\bigoplus_{i\in I}\kappa(i)_{n-1}}\right)\cdots\left(\overline{\bigoplus_{i\in I}\kappa(i)_{1}}\right)\mbox{.}$ We will show that there is a pair of exact sequence morphisms with
fixed ends $\eta\leftarrow\kappa\rightarrow0$,
which will prove (b) by \ref{thm:E=00003D0}. Indeed, by the fact
that for every $i\in I$ there is a morphism with fixed ends $\eta\mu_{i}\leftarrow\kappa_{i}$,
it follows that there is a morphism with fixed right end $\eta_{n-1}\cdots\eta_{1}\mu_{i}\leftarrow\kappa(i)_{n-1}\cdots\kappa(i)_{1}\mbox{,}$ inducing by the coproduct universal property a morphism with fixed
right end
\[
\eta_{n-1}\cdots\eta_{1}\leftarrow\left(\bigoplus_{i\in I}\kappa(i)_{n-1}\right)\cdots\left(\bigoplus_{i\in I}\kappa(i)_{1}\right)\mbox{.}
\]
Furthermore, by the proof of \ref{prop:ext1 y coprods arb} we know
that $\Psi_{1}^{-1}(\overline{\kappa'_{i}})$ has as representative
the exact sequence $\suc[B][\mbox{colim}(f_{i})][\bigoplus_{i\in I}X_{i}][f][g]\mbox{.}$ Hence, using the colimit universal property, is easy to see that there
is a morphism with fixed left end $\eta_{n}\leftarrow\Psi_{1}^{-1}(\overline{\kappa'_{i}})\mbox{.}$ Therefore, with the last morphisms we can build a morphism with fixed
ends $\eta\leftarrow\kappa\mbox{.}$ For showing the existence of a morphism with fixed ends $\kappa\rightarrow0$,
it is enough to show that $f$ is a splitting monomorphism, which
follows straightforward from the colimit universal property together
with the fact that every $f_{i}$ is a splitting monomorphism.
\item Let $(\overline{\eta_{i}})\in\prod_{i\in I}\Ext[A_{i}][n][][B]$.
We observe the following facts for every $i\in I$. Suppose $\overline{\eta_{i}}=\overline{\kappa_{n}^{i}}\cdots\overline{\kappa_{1}^{i}}$
is a natural decomposition, where 
\[
\kappa_{k}^{i}:\quad\suc[B_{k+1}^{i}][C_{k}^{i}][B_{k}^{i}]\:\forall k\in\{n,\cdots,1\}\mbox{.}
\]
Consider the coproduct canonical inclusions $u_{k}^{i}:B_{k}^{i}\rightarrow\bigoplus_{i\in I}B_{k}^{i}$.
Observe that $u_{1}^{i}=\mu_{i}\,\forall i\in I$. By \ref{cor:aE=00003DE'b}
we can see that $\overline{\left(\bigoplus_{i\in I}\kappa_{k}^{i}\right)}u_{k}^{i}=u_{k+1}^{i}\overline{\kappa_{k}^{i}}$ for all $ k\in\left\{ 1,\cdots,n+1\right\} \mbox{.}$ Hence, by \ref{prop:ext1 y coprods arb}(c), for every $i \in I$ the extension defined as $\overline{\eta}:=\Psi_{1}^{-1}(\kappa_{n}^{i})_{i\in I}\overline{\left(\bigoplus_{i\in I}\kappa^{i}{}_{n-1}\right)}\cdots\overline{\left(\bigoplus_{i\in I}\kappa^{i}{}_{1}\right)}\mbox{,}$ satisfies by recursion the following equalities 
\begin{alignat*}{1}
\overline{\eta}\mu_{i}  =\overline{\eta}u_{1}^{i} & =\Psi_{1}^{-1}(\kappa_{n}^{i})_{i\in I}\overline{\left(\bigoplus_{i\in I}\kappa^{i}{}_{n-1}\right)}\cdots\overline{\left(\bigoplus_{i\in I}\kappa^{i}{}_{1}\right)}u_{1}^{i}\\
 & =\Psi_{1}^{-1}(\kappa_{n}^{i})_{i\in I}\overline{\left(\bigoplus_{i\in I}\kappa^{i}{}_{n-1}\right)}\cdots\overline{\left(\bigoplus_{i\in I}\kappa^{i}{}_{2}\right)}u_{2}\overline{\kappa^{i}{}_{1}}\\
 & =\Psi_{1}^{-1}(\kappa_{n}^{i})_{i\in I}\overline{\left(\bigoplus_{i\in I}\kappa^{i}{}_{n-1}\right)}\cdots\overline{\left(\bigoplus_{i\in I}\kappa^{i}{}_{3}\right)}u_{3}\overline{\kappa{}_{2}^{i}}\overline{\kappa^{i}{}_{1}}\\
 & \vdots\\
 & =\Psi_{1}^{-1}(\kappa_{n}^{i})_{i\in I}\overline{\left(\bigoplus_{i\in I}\kappa^{i}{}_{n-1}\right)}u_{n-1}\overline{\kappa_{n-2}^{i}}\cdots\overline{\kappa_{1}^{i}}\\
 & =\Psi_{1}^{-1}(\kappa_{n}^{i})_{i\in I}u_{n}\overline{\kappa_{n-1}^{i}}\cdots\overline{\kappa_{1}^{i}}\\
 & =\overline{\kappa_{n}^{i}}\overline{\kappa_{n-1}^{i}}\cdots\overline{\kappa_{1}^{i}}\\
 & =\eta_{i}\mbox{.}
\end{alignat*}

\end{enumerate}
\end{proof}

By duality we have the following result.
\begin{thm}
\label{thm:ext vs prod arb} Let $\mathcal{C}$ be an Ab4{*} category, $n\geq1$, and $\{A_{i}\}_{i\in I}$
be a set of objects in $\mathcal{C}$. Consider
the product 
$\left(\pi_{i}:\prod_{i\in I}A_{i}\rightarrow A_{i}\right)_{i\in I}$.
Then, the correspondence $\Phi_{n}:\Ext[B][n][][\prod_{i\in I}A_{i}]\rightarrow\prod_{i\in I}\Ext[B][n][][A_{i}]$,
 $E\mapsto\left(\pi_{i}E\right)_{i\in I}$, is an isomorphism of abelian
groups for every $B \in \mathcal{C}$.
\end{thm}

We will end this section introducing an application related to the
tilting theory developed in recent years. Namely, R. Colpi and K.
R. Fuller developed a theory of tilting objects of projective dimension
$\leq1$ for abelian categories in \cite{colpi2007tilting}, and P.
\v{C}oupek and J. {\v{S}}t'ov{\'\i}{\v{c}}ek developed a theory of cotilting
objects of injective dimension $\leq1$ for Grothendieck categories
in \cite{vcoupek2017cotilting}. A fundamental result needed in these
theories is that 
\[
{}\mbox{Ext}_{\mathcal{A}}^{1}\left(\bigoplus_{i\in I}A_{i},X\right)=0{}\mbox{ if and only if }{}\mbox{Ext}_{\mathcal{A}}^{1}(A_{i},X)=0\,\forall i\in I\mbox{.}
\]
Such result is proved showing that, in any Ab3 abelian category $\mathcal{A}$,
there is an injective correspondence $\mbox{Ext}_{\mathcal{A}}^{1}\left(\bigoplus_{i\in I}A_{i},X\right)\rightarrow\prod_{i\in I}\mbox{Ext}_{\mathcal{A}}^{1}(A_{i},X)$ (see \cite[Proposition 8.1, Proposition 8.2]{colpi2007tilting} and
\cite[Proposition A.1]{vcoupek2017cotilting} or the proof of \ref{prop:ext1 y coprods arb}).
Now, for extending the theory to tilting objects of projective dimension
$\leq n$, it is needed a similar result for $\mbox{Ext}^{n}$. But,
it is not known in general if there is an injective correspondence
$\mbox{Ext}_{\mathcal{A}}^{n}\left(\bigoplus_{i\in I}A_{i},X\right)\rightarrow\prod_{i\in I}\mbox{Ext}_{\mathcal{A}}^{n}(A_{i},X)$.

The following result follows from \ref{thm:ext vs prod arb} and \ref{prop:extn coprods arb}.
It is worth to mention that it extends \cite[Corollary 8.3]{colpi2007tilting}
and the dual of \cite[Corollary A.2]{vcoupek2017cotilting} when the
category is Ab4. 
\begin{cor}
Let $\mathcal{C}$ be an abelian category, $n\geq1$, $\{A_{i}\}_{i\in I}$
be a set of objects in $\mathcal{C}$, and $B\in\mathcal{C}$. Then,
the following statements hold true:
\begin{enumerate}
\item If $\mathcal{C}$ is Ab4, then $\Ext[\bigoplus_{i\in I}A_{i}][n][][B]=0$
if and only if $\Ext[A_{i}][n][][B]=0\,\forall i\in I$.
\item If $\mathcal{C}$ is Ab4{*}, then $\Ext[B][n][][\prod_{i\in I}A_{i}]=0$
if and only if $\Ext[B][n][][A_{i}]=0\,\forall i\in I$.
\end{enumerate}
\end{cor}

\section{A characterization of Ab4}

This section is inspired by the comments made by Sergio Estrada during
the Coloquio Latinoamericano de \'Algebra XXIII. The goal is to prove
that if the correspondence $\Psi:\Ext[\bigoplus A_{i}][1][][B]\rightarrow\prod_{i\in I}\Ext[A_{i}][1][][B]$ 
defined above  is always  biyective
for an Ab3 category $\mathcal{C}$, then $\mathcal{C}$ is Ab4. 

Throughout this section for every natural number $ n>0$ we will consider the correspondence $\Psi_{n}:\Ext[\bigoplus X_{i}][n][][Y]\rightarrow\prod_{i\in I}\Ext[X_{i}][n][][Y]$ defined above.

In \ref{lem:n-pushouy de monos}, it was proved that, if $\mathcal{C}$
is Ab4, then given a set of exact sequences 
\[
\left\{ \eta_{i}:\:\suc[B][A_{i}][C_{i}][f_{i}]\right\} _{i\in I}\mbox{,}
\]
it can be built an exact sequence
$\suc[B][\mbox{colim}(f_{i})][\bigoplus_{i\in I}C_{i}][f]$,
where $f$ is part of the co-compatible family of morphisms associated
to $\colim[f_{i}]$. In case $\mathcal{C}$ is only an Ab3 category,
then by doing a similar construction we get an exact sequence 
$B\rightarrow\colim[f_{i}]\rightarrow\bigoplus_{i\in I}C_{i}\rightarrow0$.
Indeed, consider the direct system of exact sequences \\
\begin{minipage}[t]{1\columnwidth}%
\[
\begin{tikzpicture}[-,>=to,shorten >=1pt,auto,node distance=1cm,main node/.style=,x=1.5cm,y=1.5cm]

   \node[main node] (1) at (0,0)      {$0$};
   \node[main node] (2) [right of=1]  {$B$};
   \node[main node] (3) [right of=2]  {$B$};
   \node[main node] (4) [right of=3]  {$0$};
   \node[main node] (5) [right of=4]  {$0$};

   \node[main node] (1') [below of=1]      {$0$};
   \node[main node] (2') [right of=1']  {$B$};
   \node[main node] (3') [right of=2']  {$A_i$};
   \node[main node] (4') [right of=3']  {$C_i$};
   \node[main node] (5') [right of=4']  {$0$};

\draw[->, thin]   (1)  to node  {$$}  (2);
\draw[->, thin]   (2)  to node  {$1$}  (3);
\draw[->, thin]   (3)  to node  {$$}  (4);
\draw[->, thin]   (4)  to node  {$$}  (5);

\draw[->, thin]   (1')  to node  {$$}  (2');
\draw[->, thin]   (2')  to node  {$f_i$}  (3');
\draw[->, thin]   (3')  to node  {$$}  (4');
\draw[->, thin]   (4')  to node  {$$}  (5');

\draw[->, thin]   (2)  to node  {$1$}  (2');
\draw[->, thin]   (3)  to node  {$f_i$}  (3');
\draw[->, thin]   (4)  to node  {$$}  (4');

\end{tikzpicture}
\]%
\end{minipage}\\
Then, we have an exact sequence 
$B\stackrel{f}{\rightarrow}\colim[f_{i}]\stackrel{g}{\rightarrow}\bigoplus_{i\in I}C_{i}\rightarrow0$,
where $f$ and $g$ are induced by the colimit universal property
(see \cite[page 55]{Popescu}). Such exact sequence we shall name
it $\Theta(\eta_{i})$.

As a first step we will show that, even if the category is not Ab4,
if the correspondence $\Psi$ is biyective, then the inverse correspondence
is given by $\Theta$. That is, if $\Psi:\Ext[\bigoplus A_{i}][1][][B]\rightarrow\prod_{i\in I}\Ext[A_{i}][1][][B]$
is biyective, then for every set of exact sequences $\left\{ \eta_{i}:\:\suc[B][A_{i}][C_{i}][f_{i}]\right\} _{i\in I}$,
then the morphism $f$ in $\Phi(\eta_{i})$ is monic, and $\Psi\overline{\Theta(\eta_{i})}=1$. 
\begin{lem}
Let $\mathcal{C}$ be an AB3 category, $\left\{ A_{i}\right\} _{i\in I}$ be
a set of objects in $\mathcal{C}$, and $B\in\mathcal{C}$. Consider
the coproduct $\bigoplus_{i\in I}A_{i}$  with the canonical
inclusions $\left\{ \mu_{i}:A_{i}\rightarrow\bigoplus_{i\in I}A_{i}\right\} _{i\in I}$,
and a set of exact sequences $\left\{ \eta_{i}:\;\suc[B][E_{i}][A_{i}][f_{i}][g_{i}]\right\} _{i\in I}$.
If there is an exact sequence $\eta:\:\suc[B][E][\bigoplus_{i\in I}A_{i}][f'][g']$
such that $\overline{\eta}\mu_{i}=\overline{\eta_{i}}\,\forall i\in I$,
then the morphism $f$ in the exact sequence 
\[
\Theta(\eta_{i}):\:B\stackrel{f}{\rightarrow}\colim[f_{i}]\stackrel{g}{\rightarrow}\bigoplus_{i\in I}C_{i}\rightarrow0
\]
is a monomorphism and $\overline{\eta}=\overline{\Theta(\eta_{i})}$.\end{lem}
\begin{proof}
Consider the direct system $\left\{ f_{i}:B\rightarrow E_{i}\right\} _{i\in I}$.
We know that $\overline{\eta}\mu_{i}=\overline{\eta_{i}}\,\forall i\in I$.
Hence, for every $i\in I$ there is a morphism of exact sequences\\
\begin{minipage}[t]{0.5\columnwidth}%
\[
(1,\nu_{i},\mu_{i}):\eta_{i}\rightarrow\eta\mbox{.}
\]
Observe that the set of morphisms $\left\{ \nu_{i}:E_{i}\rightarrow E\right\} _{i\in I}$,
together with the morphism $f':B\rightarrow E$, is a co-compatible
family of morphisms.  %
\end{minipage}\hfill{}%
\fbox{\begin{minipage}[t]{0.45\columnwidth}%
\[
\begin{tikzpicture}[-,>=to,shorten >=1pt,auto,node distance=1.3cm,main node/.style=,x=1.5cm,y=1.5cm]

   \node[main node] (1) at (0,0)      {$0$};
   \node[main node] (2) [right of=1]  {$B$};
   \node[main node] (3) [right of=2]  {$E_i $};
   \node[main node] (4) [right of=3]  {$A_i $};
   \node[main node] (5) [right of=4]  {$0$};

   \node[main node] (1') [below of=1]      {$0$};
   \node[main node] (2') [right of=1']  {$B$};
   \node[main node] (3') [right of=2']  {$E$};
   \node[main node] (4') [right of=3']  {$\bigoplus A_i$};
   \node[main node] (5') [right of=4']  {$0$};

\draw[->, thin]   (1)  to node  {$$}  (2);
\draw[->, thin]   (2)  to [below] node  {$f_i$}  (3);
\draw[->, thin]   (3)  to [below]  node  {$g_i$}  (4);
\draw[->, thin]   (4)  to node  {$$}  (5);

\draw[->, thin]   (1')  to node  {$$}  (2');
\draw[->, thin]   (2')  to [above] node  {$f'$}  (3');
\draw[->, thin]   (3')  to [above] node  {$g'$}  (4');
\draw[->, thin]   (4')  to node  {$$}  (5');

\draw[-, double]   (2)  to node  {$$}  (2');
\draw[->, thin]   (3)  to node  {$\nu_i$}  (3');
\draw[->, thin]   (4)  to node  {$\mu _i$}  (4');

\end{tikzpicture}
\]
\end{minipage}}\\
Therefore, there is a unique morphism $\omega:\colim[f_{i}]\rightarrow E$
such that $\omega\sigma_{i}=\nu_{i}$ for every  $i$ in $ I$ and $\omega f=f'$,
where $\left\{ \sigma_{i}:E_{i}\rightarrow\colim[f_{i}]\right\} _{i\in I}\cup\{f:B\rightarrow\colim[f_{i}]\}$ is\\
\fbox{\begin{minipage}[t]{0.45\columnwidth}%
\[
\begin{tikzpicture}[-,>=to,shorten >=1pt,auto,node distance=1.5cm,main node/.style=,x=1.5cm,y=1.5cm]

   \node[main node] (1) at (-.5,0)      {$0$};
   \node[main node] (2) at (0,0)       {$B$};
   \node[main node] (3) [right of=2]  {$E_i $};
   \node[main node] (4) [right of=3]  {$A_i $};
   \node[main node] (5) at (2.5,0)  {$0$};

   \node[main node] (1') [below of=1]      {$$};
   \node[main node] (2') [below of=2]  {$B$};
   \node[main node] (3') [right of=2']  {$\operatorname{colim}(f_i)$};
   \node[main node] (4') [right of=3']  {$\bigoplus  A_i$};
   \node[main node] (5') [below of=5]  {$0$};

\draw[->, thin]   (1)  to node  {$$}  (2);
\draw[->, thin]   (2)  to [below] node  {$f_i$}  (3);
\draw[->, thin]   (3)  to [below] node  {$g_i$}  (4);
\draw[->, thin]   (4)  to node  {$$}  (5);

\draw[->, thin]   (2')  to node  {$f$}  (3');
\draw[->, thin]   (3')  to node  {$g$}  (4');
\draw[->, thin]   (4')  to node  {$$}  (5');

\draw[-, double]   (2)  to node  {$$}  (2');
\draw[->, thin]   (3)  to node  {$\sigma_i$}  (3');
\draw[->, thin]   (4)  to node  {$\mu _i$}  (4');

\end{tikzpicture}
\]
\end{minipage}}\hfill{}%
\begin{minipage}[t]{0.5\columnwidth}%
 the co-compatible family associated to the colimit. Notice that
$\omega f=f'$ is a monomorphism, so $f$ is also a monomorphism. 

It remains to prove that $\overline{\eta}=\overline{\Theta(\eta_{i})}$.
Observe that, by the cokernel universal property, we can build a morphism
of exact sequences%
\end{minipage}%
\\
\begin{minipage}[t]{0.5\columnwidth}%
\[
(1,\omega,\omega'):\Theta(\eta_{i})\rightarrow\eta\mbox{.}
\]
It is enough to show that $\omega'=1$. With that goal, we see that
\[
\omega'\mu_{i}g_{i}=\omega'g\sigma_{i}=g'\omega\sigma_{i}=g'\nu_{i}=\mu_{i}g_{i}\mbox{.}
\]
\end{minipage}\hfill{}%
\fbox{\begin{minipage}[t]{0.45\columnwidth}%

\[
\begin{tikzpicture}[-,>=to,shorten >=1pt,auto,node distance=1.5cm,main node/.style=,x=1.5cm,y=1.3cm]

   \node[main node] (1) at (-.5,0)      {$0$};
   \node[main node] (2) at (0,0)      {$B$};
   \node[main node] (3) [right of=2]  {$\operatorname{colim}(f_i)$};
   \node[main node] (4) [right of=3]  {$\bigoplus  A_i$};
   \node[main node] (5) at (2.7,0)  {$0$};

   \node[main node] (1') at (-.5,-1)      {$0$};
   \node[main node] (2') at (0,-1)       {$B$};
   \node[main node] (3') [right of=2']  {$E$};
   \node[main node] (4') [right of=3']  {$\bigoplus  A_i$};
   \node[main node] (5') at (2.7,-1)        {$0$};

\draw[->, thin]   (1)  to node  {$$}  (2);
\draw[->, thin]   (2)  to node [below] {$f$}  (3);
\draw[->, thin]   (3)  to node [below] {$g$}  (4);
\draw[->, thin]   (4)  to node  {$$}  (5);

\draw[->, thin]   (1')  to node  {$$}  (2');
\draw[->, thin]   (2')  to node  {$f'$}  (3');
\draw[->, thin]   (3')  to node  {$g'$}  (4');
\draw[->, thin]   (4')  to node  {$$}  (5');

\draw[-, double]   (2)  to node  {$$}  (2');
\draw[->, thin]   (3)  to node  {$\omega$}  (3');
\draw[->, thin]   (4)  to node  {$\omega '$}  (4');

\end{tikzpicture}
\]
\end{minipage}}

Hence, by the fact that $g_{i}$ is an epimorphism, $\omega'\mu_{i}=\mu_{i}\,\forall i\in I$.
Then, by the universal coproduct property we can conclude that $\omega'=1$.

\end{proof}

\begin{cor}
\label{cor:the inverse correspondence} Let $\mathcal{C}$ be an AB3
abelian category, $\left\{ A_{i}\right\} _{i\in I}$ be a set of objects
in $\mathcal{C}$, and $B\in\mathcal{C}$. Consider the coproduct
$\left\{ \mu_{i}:A_{i}\rightarrow\bigoplus_{i\in I}A_{i}\right\} _{i\in I}$,
and the correspondence $\Psi_{1}:\Ext[\bigoplus A_{i}][1][][B]\rightarrow\prod_{i\in I}\Ext[A_{i}][1][][B]$.
If $\Psi_{1}$ is biyective, then the inverse correspondence maps
each $(\overline{\eta_{i}})\in\prod_{i\in I}\Ext[A_{i}][1][][B]$,
with representatives 
\[
\eta_{i}:\;\suc[B][E_{i}][A_{i}][f_{i}]\:\forall i\in I\mbox{,}
\]
to the extension given by the exact sequence 
\[
\suc[B][\mbox{colim}(f_{i})][\bigoplus_{i\in I}A_{i}]\mbox{.}
\]
\end{cor}
\begin{thm}
\label{thm:Ab4 vs ext} Let $\mathcal{C}$ be an Ab3 category. Then,
$\mathcal{C}$ is an Ab4 category if, and only if, the correspondence
$\Psi_{1}:\Ext[\bigoplus X_{i}][1][][Y]\rightarrow\prod_{i\in I}\Ext[X_{i}][1][][Y]$
is biyective for every $Y\in\mathcal{C}$ and every set of objects
$\left\{ X_{i}\right\} _{i\in I}$.\end{thm}
\begin{proof}
By \ref{prop:ext1 y coprods arb}, it is enough to prove that if $\Psi$
is biyective for every $Y\in\mathcal{C}$ and every set of objects
$\left\{ X_{i}\right\} _{i\in I}$, then $\mathcal{C}$ is Ab4. With
this purpose, we will consider a set of exact sequences $\left\{ \eta_{i}:\:\suc[A_{i}][B_{i}][C_{i}][\alpha_{i}][\beta_{i}]\right\} _{i\in I}$
and prove that the morphism $\bigoplus_{i\in I}\alpha_{i}:\bigoplus _{i\in I}A_i \rightarrow \bigoplus _{i\in I}B_i$ is a monomorphism.
Consider the coproduct $\bigoplus_{i\in I}A_{i}$ and the canonic 
inclusions $\mu_{i}:A_{i}\rightarrow\bigoplus_{i\in I}A_{i}$. \\
\begin{minipage}[t]{0.5\columnwidth}%
 By the dual result of \ref{prop:pb:operar a izquierda}, for every $i$ in $ I$
we have an exact sequence morphism $(\mu_{i},\mu_{i}',1):\eta_{i}\rightarrow\mu_{i}\eta_{i}$,
where 
\[
\mu_{i}\eta_{i}:\:\suc[\bigoplus_{i\in I}A_{i}][E_{i}][C_{i}][f_{i}]\mbox{.}
\]
Consider the correspondence
\end{minipage}\hfill{}%
\fbox{\begin{minipage}[t]{0.45\columnwidth}%

\[
\begin{tikzpicture}[-,>=to,shorten >=1pt,auto,node distance=1.3cm,main node/.style=,x=1.5cm,y=1.5cm]

   \node[main node] (1) at (0,0)      {$0$};
   \node[main node] (2) [right of=1]  {$A_i$};
   \node[main node] (3) [right of=2]  {$B_i$};
   \node[main node] (4) [right of=3]  {$C_i$};
   \node[main node] (5) [right of=4]  {$0$};

   \node[main node] (1') [below of=1]      {$0$};
   \node[main node] (2') [right of=1']  {$\bigoplus _{i\in I} A_i$};
   \node[main node] (3') [right of=2']  {$E_i $};
   \node[main node] (4') [right of=3']  {$C_i$};
   \node[main node] (5') [right of=4']  {$0$};

\draw[->, thin]   (1)  to node  {$$}  (2);
\draw[->, thin]   (2)  to node  {$\scriptstyle \alpha _i$}  (3);
\draw[->, thin]   (3)  to node  {$\scriptstyle \beta _i $}  (4);
\draw[->, thin]   (4)  to node  {$$}  (5);

\draw[->, thin]   (1')  to node  {$$}  (2');
\draw[->, thin]   (2')  to node  {$\scriptstyle f_i$}  (3');
\draw[->, thin]   (3')  to node  {$\scriptstyle $}  (4');
\draw[->, thin]   (4')  to node  {$$}  (5');

\draw[->, thin]   (2)  to node  {$\scriptstyle \mu _i$}  (2');
\draw[->, thin]   (3)  to node  {$\scriptstyle \mu _i '$}  (3');
\draw[-, double]   (4)  to node  {$\scriptstyle $}  (4');

\end{tikzpicture}
\]
\end{minipage}}

\[
\Psi:\Ext[\bigoplus_{i\in I}C_{i}][1][][\bigoplus_{i\in I}A_{i}]\rightarrow\prod_{i\in I}\Ext[C_{i}][1][][\bigoplus_{i\in I}A_{i}]\mbox{.}
\]
By \ref{cor:the inverse correspondence}, we know that  
 $\Psi^{-1}$ maps $(\mu_{i}\overline{\eta_{i}})\in\prod_{i\in I}\Ext[C_{i}][1][][\bigoplus_{i\in I}A_{i}]$
to the extension given by the exact sequence
\[
\suc[\bigoplus_{i\in I}A_{i}][\mbox{colim}(f_{i})][\bigoplus_{i\in I}C_{i}][f]\mbox{.}
\]
\fbox{\begin{minipage}[t]{0.3\columnwidth}%

\[
\begin{tikzpicture}[-,>=to,shorten >=1pt,auto,node distance=1.5cm,main node/.style=,x=1.3cm,y=1.3cm]

\node (1) at (0,1) {$A_i$};
\node (2) at (1,1) {$B_i$};
\node (3) at (0,0) {$\bigoplus A_i$};
\node (4) at (1,0) {$E_i$};

\begin{scope}[xshift=1.5cm,yshift=0cm]
   \node (6)  at (315:1.3cm)          {$X$}  ;
\end{scope}

\draw[->, thin]  (1)  to  node  {$\alpha _i$} (2);
\draw[->, thin]  (2)  to  node  {$\mu ' _i$} (4);
\draw[->, thin]  (1)  to  node  {$\mu  _i$} (3);
\draw[->, thin]  (3)  to  node [below]  {$f  _i$} (4);
\draw[->, dashed]  (3)  to [out=300,in=180] node  {$\alpha$} (6);
\draw[->, thin]  (2)  to [out=330,in=90] node [right] {$g_i$} (6);
\draw[->, dashed]  (4)  to  node [left] {$\gamma _i$} (6);

\end{tikzpicture}
\]
\end{minipage}}\hfill{}%
\begin{minipage}[t]{0.65\columnwidth}%

We will show that $\colim[f_{i}]=\bigoplus_{i\in I}B_{i}$ to conclude
that $\bigoplus_{i\in I}\alpha_{i}$ is a monomorphism. Indeed, consider
a family of morphisms $\left\{ g_{i}:B_{i}\rightarrow X\right\} _{i\in I}$.
By the universal property of the coproduct $\bigoplus_{i\in I}A_{i}$,
there is a unique morphism $\alpha:\bigoplus_{i\in I}A_{i}\rightarrow X$
such that $g_{i}\alpha_{i}=\alpha\mu_{i}\forall i\in I$. Now, by
the universal property of the pushout on the last equality, for every
$i\in I$ there is a unique morphism $\gamma_{i}:E_{i}\rightarrow X$
such that $g_{i}=\gamma_{i}\mu_{i}'$ and $\alpha=\gamma_{i}f_{i}$.%
\end{minipage}\\

Before going further, consider the co-compatible family of morphisms
associated to the colimit $\left\{ u_{k}:E_{k}\rightarrow\colim[f_{i}]\right\} _{k\in I}$.
Observe that, by the universal property of the colimit on the last
equalities, there is a unique morphism $\Lambda:\colim[f_{i}]\rightarrow X$
such that $\Lambda u_{i}=\gamma_{i}\,\forall i\in I$ and $\Lambda f=\alpha$.
\\
\begin{minipage}[t]{0.55\columnwidth}%
In particular, if $X=\bigoplus_{i\in I}B_{i}$ and $\left\{ g_{i}:B_{i}\rightarrow\bigoplus_{i\in I}B_{i}\right\} _{i\in I}$
is the set of canonic inclusions, there is a unique morphism $\Lambda:\colim[f_{i}]\rightarrow\bigoplus_{i\in I}B_{i}$
such that $\Lambda u_{i}=\gamma_{i}\,\forall i\in I$ and $\Lambda f=\alpha$.
Furthermore, by the universal property of the coproduct $\bigoplus_{i\in I}B_{i}$,
there is a unique morphism $\Lambda':\bigoplus_{i\in I}B_{i}\rightarrow\colim[f_{i}]$
such that $\Lambda'g_{i}=u_{i}\mu'_{i}\,\forall i\in I$. 

We shall now prove that $\Lambda$ is an isomorphism and $\Lambda'=\Lambda^{-1}$.
Observe that 
\[
\Lambda\Lambda'g_{i}=\Lambda u_{i}\mu'_{i}=\gamma_{i}\mu'_{i}=g_{i}\,\forall i\in I\mbox{.}
\]
Hence, by the universal property of the coproduct $\bigoplus_{i\in I}B$,
we can conclude that $\Lambda\Lambda'=1$.

Next, we prove that $\Lambda'\Lambda=1$. Observe that %
\end{minipage}\hfill{}%
\fbox{\begin{minipage}[t]{0.4\columnwidth}%

\[
\begin{tikzpicture}[-,>=to,shorten >=1pt,auto,node distance=1.5cm,main node/.style=,x=1.5cm,y=1.5cm]

\node (1) at (0,1) {$A_i$};
\node (2) at (1,1) {$B_i$};
\node (3) at (0,0) {$\bigoplus A_i$};
\node (4) at (1,0) {$E_i$};
\node (5) at (300:1.5cm) {$\operatorname{colim}f_i$};
\node (6) [below of=5] {$X$};
\node (7) [below of=6] {$\operatorname{colim}f_i$};

\draw[->, thin]  (1)  to  node  {$\alpha _i$} (2);
\draw[->, thin]  (2)  to  node  {$\mu ' _i$} (4);
\draw[->, thin]  (1)  to  node  {$\mu  _i$} (3);
\draw[->, thin]  (3)  to  node [below]  {$f  _i$} (4);
\draw[->, thin]  (3)  to  node [below left] {$f  $} (5);
\draw[->, thin]  (4)  to  node  {$u_i  $} (5);
\draw[->, dashed]  (5)  to  node  {$\Lambda$} (6);
\draw[->, dashed]  (6)  to  node  {$\Lambda'$} (7);
\draw[->, thin]  (1)  to [out=210,in=150] node [left] {$$} (6);
\draw[->, thin]  (3)  to [out=240,in=150] node [right] {$\alpha$} (6);
\draw[->, thin]  (2)  to [out=330,in=30] node [above left] {$g_i$} (6);
\draw[->, thin]  (2)  to [out=330,in=30] node [below left] {$u_i \mu ' _i$} (7);
\draw[->, thin]  (4)  to [out=300,in=30] node [left] {$\gamma _i$} (6);

\end{tikzpicture}
\]
\end{minipage}}
\[
u_{i}\mu'_{i}\alpha_{i}=u_{i}f_{i}\mu_{i}=f\mu_{i}\,\forall i\in I,
\]
and also that
\[
u_{i}\mu'_{i}\alpha_{i}=\Lambda'g_{i}\alpha_{i}=\Lambda'\alpha\mu_{i}=\left(\Lambda'\Lambda f\right)\mu_{i}\,\forall i\in I.
\]

Hence, by the last equalities and the universal property of the coproduct
$\bigoplus_{i\in I}A_{i}$, we can conclude that $f=\Lambda'\Lambda f$.
Furthermore, observe that 
\[
(u_{i})\mu'_{i}=u_{i}\mu'_{i}\mbox{ and }\left(u_{i}\right)f_{i}=f\,\forall i\in I\mbox{,}
\]
and also that 
\[
(\Lambda'\Lambda u_{i})\mu'_{i}=\Lambda'\gamma_{i}\mu'_{i}=\Lambda'g{}_{i}=u_{i}\mu'_{i}\mbox{ and }\left(\Lambda'\Lambda u_{i}\right)f_{i}=\Lambda'\Lambda f=f\,\forall i\in I\mbox{.}
\]
Hence, by the last equalities and the universal property of the pushout
$(E_{i},f_{i},\mu'_{i})$, we can conclude that $\Lambda'\Lambda u_{i}=u_{i}$.
Now, it follows from the universal property of the colimit that $\Lambda'\Lambda=1$.
Therefore, $\Lambda$ is an isomorphism and $\Lambda'=\Lambda^{-1}$.

By the last remark, without loss of generality 
$\colim[f_{i}]=\bigoplus_{i\in I}B_{i}$, $\Lambda=1=\Lambda'$, and
$g_{i}=u_{i}\mu'_{i}\,\forall i\in I$. Now, observe that 
\[
f\mu_{i}=g_{i}\alpha_{i}\,\forall i\in I\mbox{.}
\]
Hence, by the universal property of the corpoduct $\bigoplus_{i\in I}A_{i}$,
we can conclude that $f=\bigoplus_{i\in I}\alpha_{i}$. Therefore,
$\bigoplus_{i\in I}\alpha_{i}$ is a monomorphism.
\end{proof}
We have the following equivalences.
\begin{thm}
Let $\mathcal{C}$ be an Ab3 category. Then, the following statements
are equivalent:
\begin{enumerate}
\item $\mathcal{C}$ is an Ab4 category. 
\item The correspondence $\Psi:\Ext[\bigoplus_{i\in I}X_{i}][1][][Y]\rightarrow\prod_{i\in I}\Ext[X_{i}][1][][Y]$
is biyective for every $Y\in\mathcal{C}$ and every set of objects
$\left\{ X_{i}\right\} _{i\in I}$.
\item The correspondence $\Psi_{n}:\Ext[\bigoplus_{i\in I}X_{i}][n][][Y]\rightarrow\prod_{i\in I}\Ext[X_{i}][n][][Y]$
is biyective $\forall Y\in\mathcal{C}$, every set of objects $\left\{ X_{i}\right\} _{i\in I}$,
and $\forall n>0$.
\end{enumerate}
\end{thm}
\begin{proof}
It follows from \ref{prop:extn coprods arb} and \ref{thm:Ab4 vs ext}.
\end{proof}

\begin{rem}
For an example of an  Ab3 category that is not Ab4, see the dual category of \cite[Example A.4]{vcoupek2017cotilting}.
\end{rem}
By duality we have the following result.
\begin{thm}
Let $\mathcal{C}$ be an Ab3{*} category. Then, the following statements
are equivalent:
\begin{enumerate}
\item $\mathcal{C}$ is an Ab4{*} category. 
\item The correspondence $\Psi:\Ext[Y][1][][\prod_{i\in I}X_{i}]\rightarrow\prod_{i\in I}\Ext[Y][1][][X_{i}]$
is biyective for every $Y\in\mathcal{C}$ and every set of objects
$\left\{ X_{i}\right\} _{i\in I}$.
\item The correspondence $\Psi_{n}:\Ext[Y][n][][\prod_{i\in I}X_{i}]\rightarrow\prod_{i\in I}\Ext[Y][n][][X_{i}]$
is biyective $\forall Y\in\mathcal{C}$, every set of objects $\left\{ X_{i}\right\} _{i\in I}$,
and $\forall n>0$.
\end{enumerate}
\end{thm}

\section*{Acknowledgements}

I wish to thank my advisor Octavio Mendoza for encouraging me to publish this work and for proof reading the article. I am also grateful to Sergio Estrada whose comments greatly improved the quality of this paper.

\bibliographystyle{plain}
\bibliography{yonedaext}

\end{document}